\documentclass[12pt]{amsart}
\usepackage{amssymb,amsmath,mathrsfs}
\usepackage[colorlinks=true,urlcolor=blue,
citecolor=red,linkcolor=blue,linktocpage,pdfpagelabels,
bookmarksnumbered,bookmarksopen]{hyperref}
\usepackage[active]{srcltx}
\usepackage{verbatim}
\usepackage{epsfig,graphicx,color,mathrsfs}
\usepackage{graphicx}
\usepackage{amsmath,amssymb,amsthm,amsfonts}
\usepackage{amssymb}
\usepackage[english]{babel}
\usepackage[hyperpageref]{backref}
\usepackage[left=2.7cm,right=2.7cm,top=3cm,bottom=3cm]{geometry}

\newcommand{\R}{{\mathbb R}}

\def \d {\mathrm{d}}
\def \Cp {\mathrm{Cap}}
\def \ep {\varepsilon}
\newtheorem{thm}{Theorem}[section]
\newtheorem{cor}[thm]{Corollary}
\newtheorem{lem}[thm]{Lemma}

\newtheorem{rem}[thm]{Remark}
\newtheorem{defn}[thm]{Definition}

\numberwithin{equation}{section}

\title[]{Symmetry and monotonicity of singular solutions of double phase problems}

\author[S.\,Biagi]{Stefano Biagi}
\author[F.\,Esposito]{Francesco Esposito}
\author[E.\,Vecchi]{Eugenio Vecchi}

\address[S.\,Biagi]{Politecnico di Milano - Dipartimento di Matematica
\newline\indent
Via Bonardi 9, 20133 Milano, Italy}
\email{stefano.biagi@polimi.it}

\address[F.\,Esposito]{Dipartimento di Matematica e Informatica, Universit\`a della Calabria
\newline\indent
Ponte Pietro Bucci 31B, 87036 Arcavacata di Rende, Cosenza, Italy}
\email{esposito@mat.unical.it}

\address[E.\,Vecchi]{Politecnico di Milano - Dipartimento di Matematica
\newline\indent
Via Bonardi 9, 20133 Milano, Italy}
\email{eugenio.vecchi@polimi.it}



\thanks{The authors are members of INdAM. S. Biagi
is partially supported by the INdAM-GNAMPA project
\emph{Metodi topologici per problemi al contorno associati a certe
classi di equazioni alle derivate parziali}.
F. Esposito is partially supported by PRIN project 2017JPCAPN (Italy): 
{\em Qualitative and quantitative aspects of nonlinear PDEs.}
E. Vecchi is partially supported
by the INdAM-GNAMPA project
\emph{Convergenze variazionali per funzionali e operatori dipendenti da campi vettoriali}}

\begin{document}
\begin{abstract}
We consider positive singular solutions of PDEs arising from double phase functionals. Exploiting a rather new version of the moving plane method originally developed by Sciunzi, we prove symmetry and monotonicity properties of such solutions.

\end{abstract}

\maketitle

\medskip

\section{Introduction}\label{intro}
 Integral functionals of the form
 \begin{equation}
 u \mapsto \int_{\Omega}f(\nabla u)\, \d x
 \end{equation}
 where $\Omega \subset \mathbb{R}^N$ and $f$ has {\it non-standard growth}, 
 have been object of intensive study in the Calculus of Variations since the seminal papers of Marcellini 
 \cite{marcellini1,marcellini2}. 
 Subsequently, the interest has moved towards the study of non-autonomous functionals, 
 whose prototype can be given by
 \begin{equation}\label{DoublePhase}
 u \mapsto \int_{\Omega}|\nabla u|^p + a(x) |\nabla u|^q \, \d x,
 \end{equation}
 where, in general, $1<p<q<N$ and $a(\cdot)\geq 0$. Without any aim of completeness, we refer e.g. \cite{BarColMin, ColMin1, ColMin2,ELM}.
 Functionals of the type \eqref{DoublePhase} are called
 {\it double phase functionals} and have been introduced in homogenization 
 theory by Zhikov \cite{zhikov1, zhikov2} with the aim of modeling strongly 
 anisotropic materials.
 The common trait between the aforementioned papers is the study of 
 regularity properties of minimizers, a pretty technical topic which has 
 brought to light first the deep interplay between the two exponents $p$ and $q$, 
 and second the great influence of the term $a(\cdot)$ when dealing with 
 non-autonomous functionals. In both cases, the upper bounds on 
 $q$ are strictly related to the so called {\it Lavrentiev phenomenon}, 
 see e.g. \cite{EspLeoPet}. \medskip

 In this paper we are more interested in considering the natural 
 PDE counterpart of the functional 
 \begin{equation}\label{DoublePhase2}
 u \mapsto \int_{\Omega}|\nabla u|^p + a(x) |\nabla u|^q \, \d x - F(u),
 \end{equation}
 whose Euler-Lagrange equations is given by
 \begin{equation}\label{eq:Riey}
 \left\{\begin{array}{rl}
  -\mathrm{div} \left(p|\nabla u|^{p-2}\nabla u + 
   q a(x)|\nabla u|^{q-2}\nabla u \right) = f(u) & \textrm{in } \Omega,\\
   u = 0 & \textrm{on } \partial \Omega.
  \end{array}\right.
 \end{equation}
 Here $F$ is a primitive of $f$, which is assumed to be a 
 locally Lipschitz and positive function. 
 We stress that some fundamental analytical tools like 
 weak comparison principle and summability estimates for 
 the second order derivatives of the solutions of 
 \eqref{DoublePhase2} have been recently established in \cite{riey}.
 The literature concerning existence and multiplicity results of double phase problems is rapidly growing, 
 see e.g. \cite{CoSq, LiuDai, PerSqu, GasWin} up to 
 the very recent \cite{FiPi} where double phase Kirchhoff 
 problems have been object of study by means of variational techniques.\\
 
  The aim of this work is different from the above mentioned papers and it is of classical flavour. 
  Indeed, we will prove purely qualitative properties of solutions of double phase problems of the form 
  \eqref{eq:Riey}. In particular, we are interested in generalizing some very recent 
 results contained in \cite{EMS} in order to get some symmetry 
 and monotonicity results for nontrivial solutions  
 $u \in C^{1}(\overline{\Omega}\setminus \Gamma)$
 to the following quasilinear elliptic boundary value problem
 \begin{equation}\label{eq:Riey2}
  \left\{\begin{array}{rl}
  -\mathrm{div} \left(p|\nabla u|^{p-2}\nabla u 
  + q a(x)|\nabla u|^{q-2}\nabla u \right) = f(u) & \text{in $\Omega \setminus \Gamma$},\\
  u>0  & \text{in $\Omega \setminus \Gamma$},\\
  u = 0 & \textrm{on } \partial \Omega.
  \end{array}\right.
 \end{equation}
 where $\Omega \subset \R^N$ is a bounded and smooth domain, with $N \geq 2$ and $1< p < q < N$, while $
 \Gamma \subset \Omega$ is a closed set. See below for the details.
 The solution $u$ has a possible singularity on the critical 
 set $\Gamma$ and in fact we shall only assume that $u$ is of class $C^{1}$ far from the critical set. 
 Before stating our main result, we need to properly describe 
 what a solution to equation \eqref{eq:Riey2} is.
\begin{defn} \label{def:solution} 
 We say that a function
 $u\in C^{1}({\overline\Omega}\setminus \Gamma)$ is a solution to 
 problem \eqref{eq:Riey2} if 
 it satisfies the following two properties:
 \begin{enumerate}
  \item $u > 0$ in $\Omega\setminus\Gamma$ and $u=0$ on $\partial \Omega$;
  \item for every $\varphi\in
	C^1_c(\Omega\setminus\Gamma)$ one has
	\begin{equation} \label{debil1}
	\int_\Omega \big(p|\nabla u|^{p-2}+qa(x)|\nabla u|^{q-2}\big)
	\langle \nabla u, \nabla
	\varphi\rangle\,\d x\, = \,\int_\Omega f(u)\varphi\,\d x,
	\end{equation}
 \end{enumerate}	
 \noindent where $\langle \cdot, \cdot \rangle$ denotes the standard scalar product in $\mathbb{R}^N$.
\end{defn}
 Now we state our main result.
 \begin{thm}\label{thm:symmetry}
  Let $q > p \geq 2$  and let $\Omega\subseteq\R^N$ be a convex open set, 
  symmetric with respect to the  $x_1$-direction. Moreover, let $\Gamma\subseteq\Omega\cap\{x_1 = 0\}$
  be a closed set such that
  $$\mathrm{Cap}_q(\Gamma) = 0.$$
  Finally, we assume that the following `structural' assumptions are satisfied:
  \begin{enumerate}
   \item $a\in L^\infty(\Omega)\cap C^1(\Omega)$ is \emph{non-negative and independent of $x_1$}
   
   \item $f:\R \rightarrow \R$ is a locally Lipschitz continuous function with $f(s)>0$ for $s>0$.
  \end{enumerate}  
  Then, any solution $u\in C^1(\overline{\Omega}\setminus\Gamma)$ 
  to \eqref{eq:Riey2}
  is symmetric wrt the hyperplane $\{x_1=0\}$ and increasing 
  in the $x_1$-direction in $\Omega \cap \{x_1<0\}$.
\end{thm}
\medskip

We stress that for $a \equiv p/q$, the problem
\eqref{DoublePhase2} reduces to 
 \begin{equation}\label{pqLap}
  \begin{cases}
   -\Delta_{p} u -\Delta_{q} u =  f(u) & \text{in $\Omega \setminus \Gamma$}\\
    u>0  & \text{in $\Omega \setminus \Gamma$}    \\ 
    u=0 & \text{on $\partial \Omega$}
  \end{cases}
 \end{equation}
 \noindent where the operator appearing on the left hand-side is called $(p,q)-Laplacian$. Obviously, by 
 taking $a \equiv 0$ our result boils down to the $p$--Laplacian case considered in \cite{EMS}.

 We notice that in the planar case $N=2$, $\Gamma$ reduces to a point: in this case, Theorem 
 \ref{thm:symmetry} can be seen as a generalization of \cite{Ter, CLN2} to the double phase setting. The 
 case of {\it point singularity} for cooperative elliptic systems has been considered in \cite{BVV0}.\\

 Let us now spend a few comments on the Theorem \ref{thm:symmetry}. 
 The technique that we will develop to prove Theorem \ref{thm:symmetry} is a quite recent version of the  
 moving plane method introduced by Sciunzi in \cite{Dino} in order to deal with singular solutions of 
 semilinear elliptic problems driven by the classical Laplacian operator. The technique is so powerful and 
 flexible that has been recently extended to cover the case of unbounded sets \cite{EFS}, the $p$--
 Laplacian operator \cite{MPS}, cooperative elliptic systems \cite{BVV,esposito}, the fractional Laplacian 
 \cite{MPS} and mixed local--nonlocal elliptic operators \cite{BDVV}.
 We want to stress that the technique we will use to prove Theorem \ref{thm:symmetry} actually works for 
 any $q > p \geq 2$. Nevertheless, the result is {\it more meaningful} if stated for $2 \leq p < q \leq N$ 
 because there are no sets of zero $q$-capacity when $q > N$, see e.g. \cite{HeinBook}. 
 We also want to highlight that our result holds for $q> p \geq 2$;  this
 lower bound for $p$ is somehow necessary for a quite technical reason. Indeed, if  $1<p<2$ the gradient of 
 the solution may blows up near the critical set $\Gamma$ and hence the inverse of the weight $\rho:=(|
 \nabla u|+|\nabla u_\lambda|)^{p-2}$ may not have the right summability properties 
 (see Remark \ref{rem.Gammanonempty} for more details). 
 This issue already occurs when dealing with the $p$--Laplacian, see \cite{EMS}. In that case however, the 
 authors made an 
 accurate analysis of the behaviour of the gradient of the solution near the set $\Gamma$ based on previous  
 results contained in \cite{PQS}: whether a similar approach could be fruitful 
 in our setting is currently 
 an open problem and will be the aim of future projects. \\

 It is clear from our previous considerations that the lower bound on $p$ (i.e. $p\geq 2$) may be avoided 
 if $\Gamma = \emptyset$. To the best of our knowledge, Theorem \ref{thm:symmetry} is new even in this 
 simpler setting. In this case, 
 symmetry and monotocity properties of the solutions hold true for every $q>p>1$, and
 the proof can be performed by using the classical moving plane method introduced by 
 Alexandrov \cite{A} and Serrin \cite{serrin}, and subsequently improved in the celebrated papers  
 \cite{GNN} and \cite{BN} in the context of semilinear elliptic equations. Since then, the literature 
 relative to generalizations of these result to more and more general situations has become so huge that we 
 do not even attempt at recalling all the contributions. We limit ourselves to \cite{Troy, Dan2, 
 Dan, CoVe1, CoVe2} for the case of cooperative elliptic systems (also on unbounded/non smo\-oth domains)
 and for the case of the composite plate problem. Finally, we refer to \cite{Da2, DP, DS1, ES, FMRS, FMS, 
 FMS2, MSS} for the  quasilinear case, which is very close to our present needs.  
 For completeness we state the following:
 
\begin{thm}\label{thm:noSingularSym}
 Let $p,q>1$  and let $\Omega\subseteq\R^N$ be a convex open set, symmetric with respect to the  $x_1$-
 direction. Moreover, let us assume that the following `structural' assumptions are satisfied:
\begin{enumerate}
	\item $a\in L^\infty(\Omega)\cap C^1({\Omega})$ is \emph{non-negative and independent of $x_1$}
	
	\item $f:\R \rightarrow \R$ is a locally Lipschitz continuous function with $f(s)>0$ for $s>0$.
\end{enumerate}  
Then, any solution 
$u\in C^1(\overline{\Omega})$ to \eqref{eq:Riey2} with $\Gamma = \emptyset$ is symmetric wrt 
the hyperplane $\{x_1=0\}$ and increasing in the $x_1$-direction in $\Omega \cap \{x_1<0\}$. 
\end{thm}

Before closing the Introduction, we must comment on the assumption made on the term $a(x)$. We believe that 
the requirement of being independent of $x_1$ is a merely technical issue strictly related to the moving 
plane method. Indeed, our assumption on $a$ appears also in \cite{Pol} where the authors prove symmetry 
results for nonnegative solutions of fully nonlinear operators. We plan to come back to the possibility of 
removing such an assumption in a future paper.
\medskip
 
 The plan of the paper is the following:

\begin{itemize}
	\item[-] In Section \ref{sec.preliminaries} we fix the notations 
	used in all the paper. Moreover, we recall the notion of $r$-capacity and some related theorems. 
	Finally, we prove Lemma \ref{leaiuto} and Lemma \ref{conncomp} 
	that are two key ingredients in order to 
	apply the moving plane procedure.
	
	\item[-] In Section \ref{sec.mainresult1} we prove Theorem \ref{thm:symmetry} performing the moving 
	plane technique in the $x_1$-direction and using the results stated in Section \ref{sec.preliminaries}.
	
	\item[-] In Section \ref{sec.mainresult2}  we prove Theorem \ref{thm:noSingularSym} 
	using some results 
	contained in \cite{EMS, riey} and performing the moving plane method 
	in a standard way (since $\Gamma = 
	\emptyset$).
	
	\item[-] In the Appendix we state two essential results: a strong comparison principle and a Hopf-type 
	lemma that apply specifically to our context.
\end{itemize}

\section{Notations and auxiliary results} \label{sec.preliminaries}
 The aim of this section is twofold: on the one hand, we fix once and for all
 the relevant notations used throughout the paper; on the other hand,
 we present some auxiliary results which shall be key ingredients
 for the proof of Theorem \ref{thm:symmetry}.
 
 \subsection{A review of $r$-capacity.} 
 Let $1\leq r \leq N$ be fixed, and let $K\subseteq\R^N$ be a \emph{compact set}.
 We remind that the $r$-capacity of $K$ is defined as
 \begin{equation}\label{eq:q-capacity}
 \mathrm{Cap}_r(K):=
  \inf \bigg\{\int_{\mathbb{R}^N}|\nabla \varphi|^r\,\d x\,:\,
  \text{$\varphi \in C^{\infty}_c(\R^N)$ and $\varphi\geq 1$ on $K$}\bigg\}.
 \end{equation}
 Moreover, if $D\subseteq\R^N$ is any bounded set containing $K$,
 it is possible to define the $r$-capacity of \emph{the condenser
 $(K,D)$} in the following way
 \begin{equation}\label{eq:q-capacityrel}
 \mathrm{Cap}^D_r(K) :=
  \inf \bigg\{\int_{\mathbb{R}^N}|\nabla \varphi|^r\,\d x\,:\,
  \text{$\varphi \in C^{\infty}_c(D)$ and $\varphi\geq 1$ on $K$}\bigg\}.
 \end{equation}
 As already described in the Introduction,
 the main aim of this paper is to investigate symmetry/monotonicity properties
 of the solutions to \eqref{eq:Riey2}, which may present
 singularities on the (compact) set $\Gamma$. Since the key assumption on
 $\Gamma$ is that
 $$\mathrm{Cap}_q(\Gamma) = 0$$
 (broadly put, $\Gamma$ has to be `small enough'),
 it is worth reviewing some basic facts about compact sets with vanishing
 capacity. In what follows, we denote 
  by
 $\mathcal{H}^d(\cdot)$ the standard $d$-dimensional Hausdorff measure
 on $\R^N$, as defined, e.g., in \cite{evans}.
 \begin{thm} \label{thm.propCp}
 The following assertions hold true.
 \begin{enumerate}
   \item If $\Cp_r(K) = 0$, then $\Cp_r^D(K) = 0$ for any bounded set $D\supseteq K$.
   \vspace*{0.05cm}
   
   \item If $\Cp_r(K) = 0$, then $\mathcal{H}^s(K) = 0$ for every $s > N - r$.
   \vspace*{0.05cm}
   
   \item If $\mathcal{H}^{N-r}(K) < \infty$, then $\Cp_r(K) = 0$.
\end{enumerate}  
\end{thm}
 For a complete proof of Theorem \ref{thm.propCp}, 
 we refer the Reader to \cite[Sec.\,2.24]{HeinBook}.
 \begin{cor} \label{cor.pqCap}
  Let $1\leq p < q \leq N$ and let $K\subseteq\R^N$ be compact.
  Then,
  $$\Cp_q(K) = 0\,\,\Longrightarrow\,\,\Cp_p(K) = 0.$$
 \end{cor}
 \begin{proof}
  Since, by assumption, $\Cp_q(K) = 0$, by Theorem \ref{thm.propCp}-(2)
  we have $\mathcal{H}^s(K) = 0$ for every $s > N-q$;
  in particular, as $p < q$, we derive that
  $$\mathcal{H}^{N-p}(K) = 0.$$
  Using this fact and Theorem \ref{thm.propCp}-(3), we then
  conclude that $\Cp_p(K) = 0$.
 \end{proof}
 On account of Corollary \ref{cor.pqCap}, if $\Gamma\subseteq\R^N$
 is as in Theorem \ref{thm:symmetry} we have
 $$\Cp_p(\Gamma) = 0.$$
 \subsection{Notations for the moving plane method.}
 Let $\Gamma\subseteq\Omega\subseteq\R^N$ be as in the statement
 of Theorem \ref{thm:symmetry}, and let
 $u\in C^1(\overline{\Omega}\setminus\Gamma)$ be a solution
 of \eqref{eq:Riey2}. For any fixed $\lambda\in\R$, we in\-di\-ca\-te by $R_\lambda$
 the reflection trough the hy\-per\-pla\-ne $\Pi_\lambda := \{x_1 = \lambda\}$,
 that is,
 \begin{equation} \label{eq.defRlambda}
  R_\lambda(x) = x_\lambda := (2\lambda-x_1,x_2,\ldots,x_N)
  \qquad (\text{for all $x\in\R^N$});
 \end{equation}
 accordingly, we define the function
 \begin{equation} \label{eq.defulambda}
  u_\lambda(x) := u(x_\lambda), \qquad\text{for all
  $x\in R_\lambda\big(\overline{\Omega}\setminus\Gamma\big)$}.
 \end{equation}
 We point out that, since $u$ solves \eqref{eq:Riey2}
 and $a$ is \emph{independent of $x_1$}, one has
 \begin{enumerate}
  \item  
 $u_\lambda\in C^1(R_\lambda(\overline{\Omega}\setminus\Gamma))$;
  \item $u_\lambda > 0$ in $R_\lambda(\Omega\setminus\Gamma)$ and 
  $u_\lambda \equiv 0$ on $R_\lambda(\partial\Omega\setminus\Gamma)$;
  \item for every test function $\varphi\in C^1_c(R_\lambda(\Omega\setminus\Gamma))$
  one has
  \begin{equation} \label{eq.PDEulambda}
  \int_{R_\lambda(\Omega)}\big( p|\nabla u_\lambda|^{p-2}+qa(x)|\nabla u_\lambda|^{q-2}\big)
  \langle \nabla u_\lambda,\nabla \varphi\rangle\,\d x\,=
  \,\int_{R_\lambda(\Omega)}f(u_\lambda)\varphi\,\d x.
  \end{equation}
 \end{enumerate}
 To proceed further, we let
 \begin{equation} \label{eq.defaOmega}
  \mathbf{a} = \mathbf{a}_\Omega := \inf_{x\in\Omega}x_1
 \end{equation}
 and we observe that, since $\Omega$ is (bounded and) symmetric with respect to
 the $x_1$-direction,
 we certainly have $-\infty < \mathbf{a} < 0$.
 Hence, for every $\lambda\in(\mathbf{a},0)$ we can set
 \begin{equation} \label{eq.defOmegalambda}
 \Omega_\lambda := \{x\in\Omega:\,x_1<\lambda\}.
 \end{equation}
 Notice that the convexity of
 $\Omega$ in the $x_1$-direction ensures that
 \begin{equation} \label{eq.inclusionOmegalambda}
  \Omega_\lambda\subseteq R_\lambda(\Omega)\cap \Omega.
 \end{equation}
 Finally, for every $\lambda\in(\mathbf{a},0)$ we define the function
 $$w_\lambda(x) := (u-u_\lambda)(x), \qquad \text{for
 $x\in (\overline{\Omega}\setminus\Gamma)\cap R_\lambda(\overline{\Omega}\setminus\Gamma)$}.$$
 On account of \eqref{eq.inclusionOmegalambda}, 
 $w_\lambda$ is surely well-posed on
 $\overline{\Omega}_\lambda\setminus R_\lambda(\Gamma)$.
 
 \subsection{Auxiliary results.}
 From now on, we assume that all the hypotheses of Theorem \ref{thm:symmetry}
 are satisfied. Moreover, we tacitly inherit all the notations
 introduced so far.
 \medskip
 
 To begin with, we remind some
 identities between vectors in $\R^N$ which
 are very useful in dealing
 with quasilinear operators:
 \emph{for every $s> 1$ 
 there exist constants $C_1,C_2,C_3 > 0$, only depending on $s$, such that,
 for every $\eta,\eta'\in\R^N$, one has}
 \begin{equation}\label{eq:inequalities}
  \begin{split}
 & \langle |\eta|^{s-2}\eta-|\eta'|^{s-2}\eta', \eta- \eta'\rangle \geq C_1
(|\eta|+|\eta'|)^{s-2}|\eta-\eta'|^2, \\[0,15cm]
& \big| |\eta|^{s-2}\eta-|\eta'|^{s-2}\eta '| \leq C_2
(|\eta|+|\eta'|)^{s-2}|\eta-\eta '|,\\[0.15cm]
& \langle |\eta|^{s-2}\eta-|\eta'|^{s-2}\eta ', \eta-\eta ' \rangle \geq C_3
|\eta-\eta '|^s \qquad (\text{if $s\geq 2$}), \\[0.15cm]
& \big| |\eta|^{s-2}\eta-|\eta'|^{s-2}\eta '| \leq C_4|\eta-\eta'|^{s-1}
\qquad (\text{if $1<s<2$}).
\end{split}
\end{equation}
We refer, e.g., to \cite{Da2} for a proof of 
\eqref{eq:inequalities}. \medskip

 Next, we need to define
 an \emph{ad-hoc} family of Sobolev functions in $\Omega$ allowing
 us to `cut off' of the singular set $\Gamma$. To this end, let
 $\varepsilon >0$ be small enough and let 
 $$\mathcal{B}^{\lambda}_{\epsilon}
 := \big\{x\in\R^N:\,\mathrm{dist}(x, R_{\lambda}(\Gamma)) < \varepsilon\big\}.$$ 
 Since $R_\lambda$ is an affine map, it is easy to see that
 $$\Cp_q\big(R_\lambda(\Gamma)\big) = 0;$$
 as a consequence, by Theorem \ref{thm.propCp}-(1) there exists
 $\varphi_{\varepsilon} \in
 C^{\infty}_c(\mathcal{B}^{\lambda}_{\varepsilon})$ such that
 \begin{equation} \label{eq.choicephieps}
  \text{$\varphi_\varepsilon\geq 1$ on $R_\lambda(\Gamma)$}\qquad
   \text{and}\qquad
   \int_{\mathcal{B}^{\lambda}_{\varepsilon}} |\nabla
  \varphi_{\varepsilon}|^q\,\d x < \varepsilon.
 \end{equation}
 We then consider the Lipschitz functions
 \begin{itemize}
  \item $T(s) := \max\{0;\min\{s;1\}\}$ (for $s\in\R$), \vspace*{0.03cm}
  \item $g(t) := \max\{0;-2s+1\}$ (for $t\geq 0$)
 \end{itemize}
 and we define, for $x\in\R^N$,
 \begin{equation} \label{test1}
  \psi_{\varepsilon}(x):= g\big(T(\varphi_{\varepsilon}(x))\big).
 \end{equation}
 In view of \eqref{eq.choicephieps}, and taking into account
 the very definitions of $T$ and $g$, it is not difficult to recognize that
 $\psi_\varepsilon$ satisfy the following properties:
 \begin{enumerate}
  \item $\psi_\varepsilon \equiv 1$ on $\R^N\setminus \mathcal{B}^\lambda_\varepsilon$
  and $\psi_\varepsilon\equiv 0$ on some neighborhood of $R_\lambda(\Gamma)$, say
  $\mathcal{V}^\lambda_\ep$; 
  \vspace*{0.03cm}
  
  \item $0\leq \psi_\varepsilon\leq 1$ on $\R^N$;
  \vspace*{0.03cm}
  
  \item $\psi_\varepsilon$ is Lipschitz-continuous in $\R^N$, so that
  $\psi_\varepsilon\in W^{1,\infty}(\R^N)$;
  \vspace*{0.03cm}
  
  \item there exists a constant $C > 0$, independent of $\varepsilon$, such that
  \begin{equation} \label{eq.estimNablapsi}
   \int_{\R^N}|\nabla
  \psi_{\varepsilon}|^q\,\d x \leq C \varepsilon.
  \end{equation}
 \end{enumerate}
 In particular, by combining (1), \eqref{eq.estimNablapsi} and H\"older's inequality
 we get
 \begin{equation} \label{eq.nablaPsiepHolder}
  \int_{\R^N}|\nabla\psi_\varepsilon|^r\,\d x 
  =
  \int_{\mathcal{B}^\lambda_\varepsilon}|\nabla \psi_\varepsilon|^r\,\d x\leq 
  C'\varepsilon^{r/q}\qquad
  \text{for every $1\leq r < q$},
 \end{equation}
 where $C' > 0$ is a constant which can be chosen
 independently of $\varepsilon$.
 \medskip
 
 With the family $\{\psi_\varepsilon\}_\varepsilon$ at hand, we can prove the following
 key lemma.
 \begin{lem}\label{leaiuto}
   For any fixed $\lambda\in(\mathbf{a},0)$ we have
	\begin{equation} \label{eq.lemmaSumm}
	\int_{\Omega_\lambda} \big(p(|\nabla u| + |\nabla u_\lambda|)^{p-2}+
	qa(x)(|\nabla u| + |\nabla u_\lambda|)^{q-2}\big)\cdot|\nabla
	w_\lambda^+|^2\,\d x\leq
	\mathbf{c}_0,
	\end{equation}
	where $\mathbf{c}_0 > 0$ is a constant only depending on $p,q,\lambda$ and
	$\|u\|_{L^\infty(\Omega_\lambda)}$.
 \end{lem}
 \begin{proof}
  For every fixed $\varepsilon > 0$, we consider the function
  $$\varphi_\varepsilon(x):= 
  \begin{cases}
   w_\lambda^+(x)\,\psi_\varepsilon^{p+q}(x) = (u-u_\lambda)^+(x)\,\psi_\varepsilon^{p+q}(x),
   & \text{if $x\in\Omega_\lambda$}, \\
   0, & \text{otherwise}.
   \end{cases}$$
  We claim that the following assertions hold:
  \begin{itemize}
   \item[(i)] $\varphi_\ep\in \mathrm{Lip}(\R^N)$;
   \item[(ii)]$\mathrm{supp}(\varphi_\ep)\subseteq\Omega_\lambda$ and
  $\varphi_\ep\equiv 0$ near $R_\lambda(\Gamma)$.
  \end{itemize}
  In fact, since $u\in C^1(\overline{\Omega}_\lambda)$ and $u_\lambda\in C^1(
  \overline{\Omega}_\lambda\setminus R_\lambda(\Omega))$,
  we have 
  $w_\lambda^+\in \mathrm{Lip}(\overline{\Omega}_\lambda\setminus V)$
  \emph{for every o\-pen set $V\supseteq R_\lambda(\Gamma)$};
  as a consequence,
  reminding that $\psi_\varepsilon\in\mathrm{Lip}(\R^N)$
  and $\psi_\varepsilon\equiv 0$ on a ne\-igh\-bor\-hood of $R_\lambda(\Gamma)$,
  we get
  $\varphi_\varepsilon\in \mathrm{Lip}(\overline{\Omega}_\lambda)$.
  On the other hand, since 
  $\varphi_\varepsilon\equiv 0$ on $\partial\Omega_\lambda$, we easily conclude that
  $\varphi_\ep\in \mathrm{Lip}(\R^N)$,
  as claimed. As for assertion (ii), it is a direct consequence
  of the very definition of $\varphi_\ep$ and of the fact that
  $$\text{$\psi_\ep\equiv 0$ on $\mathcal{V}^\lambda_\ep\supseteq R_\lambda(\Gamma)$}.$$
  On account of properties (i)-(ii) of $\varphi_\ep$,
   a standard density argument allows us to use
  $\varphi_\ep$ as a test function \emph{both in \eqref{debil1} and
  \eqref{eq.PDEulambda}}; reminding that $a$ is independent of $x_1$, this gives
	\begin{equation*}
	\begin{split} 
     &
     p\int_{\Omega_\lambda}
     \langle |\nabla u|^{p-2} \nabla u - |\nabla u_\lambda|^{p-2}\nabla u_\lambda,\nabla
     \varphi_\ep\rangle\,\d x
     \\
     & \qquad\qquad+ q
     \int_{\Omega_\lambda}a(x)\,\langle
     |\nabla u|^{q-2} \nabla u - |\nabla u_\lambda|^{q-2}\nabla u_\lambda,
		\nabla \varphi_\ep\rangle\,\d x \\
		&\quad = \int_{\Omega_\lambda} (f(u)-f(u_\lambda))\varphi_\ep\,\d x.
	\end{split}
  \end{equation*}
  By unraveling the very definition of $\varphi_\ep$, we then obtain
  \begin{equation} \label{eq.identityIntermediate}
	\begin{split} 
     &
     p\int_{\Omega_\lambda}
     \psi_\ep^{p+q}\cdot
     \langle |\nabla u|^{p-2} \nabla u - |\nabla u_\lambda|^{p-2}\nabla u_\lambda,\nabla
     w_\lambda^+\rangle\,\d x
     \\
     & \qquad\quad+ q
     \int_{\Omega_\lambda}\psi_\ep^{p+q}\cdot a(x)\,\langle
     |\nabla u|^{q-2} \nabla u - |\nabla u_\lambda|^{q-2}\nabla u_\lambda,
		\nabla w_\lambda^+\rangle\,\d x \\
    & \qquad\quad
    + p(p+q)
    \int_{\Omega_\lambda}
     w_\lambda^+\cdot\psi_\ep^{p+q-1}\cdot
     \langle |\nabla u|^{p-2} \nabla u - |\nabla u_\lambda|^{p-2}\nabla u_\lambda,\nabla
     \psi_\ep\rangle\,\d x \\
     & \qquad\quad
    + q(p+q)
    \int_{\Omega_\lambda}w_\lambda^+\cdot\psi_\ep^{p+q-1}\cdot a(x)\,\langle
     |\nabla u|^{q-2} \nabla u - |\nabla u_\lambda|^{q-2}\nabla u_\lambda,
		\nabla w_\lambda^+\rangle\,\d x \\
		& \quad = \int_{\Omega_\lambda} (f(u)-f(u_\lambda))\,w_\lambda^+\,\psi_\ep^{p+q}\,\d x.
	\end{split}
  \end{equation}
  We now observe that the integral in the
  left-hand side of \eqref{eq.identityIntermediate} is actually
  performed on the set
  $\mathcal{O}_\lambda := \{x\in\Omega_\lambda:\,u\geq u_\lambda\}\setminus R_\lambda(\Gamma)$;
  moreover, for every $x\in\mathcal{O}_\lambda$ we have
  $$0\leq u_\lambda(x)\leq u(x)\leq \|u\|_{L^\infty(\Omega_\lambda)}.$$
  As a consequence, since $f$ is locally Lipschitz-continuous on $\R$, we have
  \begin{equation} \label{eq.estimfLipschitz}
   \begin{split}
   & \int_{\Omega_\lambda} (f(u)-f(u_\lambda))\,w_\lambda^+\,\psi_\ep^{p+q}\,\d x
   = \int_{\Omega_\lambda} \frac{f(u)-f(u_\lambda)}{u-u_\lambda}\,(w_\lambda^+)^2\,\psi_\ep^{p+q}\,\d x
   \\
   & \qquad\qquad\leq L\int_{\Omega_\lambda}
   (w_\lambda^+)^2\,\psi_\ep^{p+q}\,\d x,
   \end{split}
  \end{equation}
  where $L = L(f,u,\lambda)> 0$ 
  is the Lipschitz constant of $f$ on 
  the interval $[0,\|u\|_{L^\infty(\Omega_\lambda)}]\subseteq\R$.
  Using \eqref{eq.estimfLipschitz} and the estimates in \eqref{eq:inequalities},
  from \eqref{eq.identityIntermediate} we then obtain
  \begin{equation} \label{eq.tostartfrom}
   \begin{split}
   & C_1\int_{\Omega_\lambda}\psi_\ep^{p+q}\,\big\{
   p(|\nabla u|+|\nabla u_\lambda|)^{p-2}
   + q a(x)\,(|\nabla u|+|\nabla u_\lambda|)^{q-2}\big\}
   \cdot |\nabla w_\lambda^+|^2\,\d x \\
   & \qquad 
   \leq p\int_{\Omega_\lambda}
     \psi_\ep^{p+q}\,
     \langle |\nabla u|^{p-2} \nabla u - |\nabla u_\lambda|^{p-2}\nabla u_\lambda,\nabla
     w_\lambda^+\rangle\,\d x
     \\
     & \qquad\quad+ q
     \int_{\Omega_\lambda}\psi_\ep^{p+q}\, a(x)\,\langle
     |\nabla u|^{q-2} \nabla u - |\nabla u_\lambda|^{q-2}\nabla u_\lambda,
		\nabla w_\lambda^+\rangle\,\d x \\
    & \qquad \leq p(p+q)
    \int_{\Omega_\lambda}
     w_\lambda^+\,\psi_\ep^{p+q-1}\cdot
     \big| |\nabla u|^{p-2} \nabla u - |\nabla u_\lambda|^{p-2}\nabla u_\lambda\big|\,
     |\nabla
     \psi_\ep|\,\d x \\
     & \qquad\quad
    + q(p+q)
    \int_{\Omega_\lambda}w_\lambda^+\,\psi_\ep^{p+q-1}\cdot a(x)\,\big|
     |\nabla u|^{q-2} \nabla u - |\nabla u_\lambda|^{q-2}\nabla u_\lambda
     \big|\,|\nabla w_\lambda^+|\,\d x \\
     & \qquad\quad 
     +L_f\,\int_{\Omega_\lambda}
   (w_\lambda^+)^2\,\psi_\ep^{p+q}\,\d x \\
   & \qquad \leq C_0\,\bigg(I_p+I_q+\int_{\Omega_\lambda}
   (w_\lambda^+)^2\,\psi_\ep^{p+q}\,\d x\bigg),
   \end{split}
  \end{equation}
  where $C_0 = C_0(p,q,\lambda,\|u\|_{L^\infty(\Omega)},f) > 0$
  is a suitable constant and
  \begin{equation} \label{eq.notationIpIq}
   \begin{split}
   & I_p := \int_{\Omega_\lambda}
   w_\lambda^+\,\psi_\ep^{p+q-1}\cdot
   (|\nabla u|+|\nabla u_\lambda|)^{p-2}|\nabla w_\lambda^+|\,
   |\nabla \psi_\ep|\,\d x, \\[0.1cm]
   &  I_q := \int_{\Omega_\lambda}
   w_\lambda^+\,\psi_\ep^{p+q-1}\cdot
   (|\nabla u|+|\nabla u_\lambda|)^{q-2}|\nabla w_\lambda^+|\,
   |\nabla \psi_\ep|\,\d x.
   \end{split}
  \end{equation}
  In order to complete the proof, we start from 
  \eqref{eq.tostartfrom} and we provide
  an e\-sti\-ma\-te of both $I_p$ and $I_q$. 
  Actually, we limit ourselves to consider $I_p$, since 
  $I_q$ can be treated analogously.
  \vspace*{0.08cm}
  
   To begin with, we split the set $\Omega_\lambda$ as 
  $\Omega_\lambda = \Omega^{(1)}_\lambda\cup\Omega^{(2)}_\lambda$, where
  \begin{align*}
  & \Omega^{(1)}_\lambda = \{x\in\Omega_\lambda\setminus R_\lambda(\Gamma):
   \,|\nabla u_\lambda(x)|
   < 2|\nabla u|\}\qquad\text{and} \\[0.08cm]
  & \qquad \Omega^{(2)}_\lambda = \{
   x\in\Omega_\lambda\setminus R_\lambda(\Gamma):\,|\nabla u_\lambda(x)|
   \geq 2|\nabla u|\};
   \end{align*}
  accordingly, 
  since Theorem \ref{thm.propCp}-(2) ensures that 
  $\mathcal{H}^N(R_\lambda(\Gamma)) = 0$,  we write 
  $$I_p = I_{p,1}+I_{p,2},\qquad\text{with
  $I_{p,i} = \int_{\Omega^{(i)}_\lambda}
   \{\cdots\}\,\d x \quad(i = 1,2)$}.$$
  We then proceed by estimating $I_{p,1},\,I_{p,2}$ separately.
  \medskip
  
  \textsc{Step I: Estimate of $I_{p,1}$}. By definition, for every
  $x\in \Omega_\lambda^{(1)}$ we have
  \begin{equation} \label{eq.estimInOmegafirst}
   |\nabla u_\lambda(x)|+|\nabla u(x)| < 3|\nabla u(x)|;
  \end{equation}
  Using the weighted
  Young inequality and \eqref{eq.estimInOmegafirst}, for every $\rho > 0$ we get
  \begin{align*}
   I_{p,1} & \leq
   \frac{\rho}{2}\int_{\Omega^{(1)}_\lambda}
   (|\nabla u|+|\nabla u_\lambda|)^{p-2}|\nabla w_\lambda^+|^2\,\psi_\ep^{p+q}\,\d x
   \\
   & \qquad +\frac{1}{2\rho}
   \int_{\Omega^{(1)}_\lambda}
   (|\nabla u|+|\nabla u_\lambda|)^{p-2}|\nabla \psi_\ep|^2\,
   (w_\lambda^+)^2\,\psi_\ep^{p+q-2}\,\d x \\
   & \leq 
   \frac{\rho}{2}\int_{\Omega^{(1)}_\lambda}
   \psi_\ep^{p+q}\,(|\nabla u|+|\nabla u_\lambda|)^{p-2}|\nabla w_\lambda^+|^2\,\d x
   \\
   & \qquad+\frac{3^{p-2}}{2\rho}
   \int_{\Omega^{(1)}_\lambda}
   |\nabla u|^{p-2}|\nabla \psi_\ep|^2\,
   (w_\lambda^+)^2\,\psi_\ep^{p+q-2}\,\d x = (\bigstar);
  \end{align*}
  from this, reminding that
  $0\leq\psi_\ep\leq 1$ and using H\"older's inequality, we have
  \allowdisplaybreaks 
  \begin{align*}
   (\bigstar) & \leq 
   \frac{\rho}{2}\int_{\Omega^{(1)}_\lambda}
   \psi_\ep^{p+q}\,(|\nabla u|+|\nabla u_\lambda|)^{p-2}|\nabla w_\lambda^+|^2\,\d x
   \\
   & \qquad + \frac{3^{p-2}}{2\rho}
   \int_{\Omega^{(1)}_\lambda}
   |\nabla u|^{p-2}|\nabla \psi_\ep|^2\,
   (w_\lambda^+)^2\,\psi_\ep^{p-2}\,\d x \\
   & \leq 
   \frac{\rho}{2}\int_{\Omega^{(1)}_\lambda}
   \psi_\ep^{p+q}\,(|\nabla u|+|\nabla u_\lambda|)^{p-2}|\nabla w_\lambda^+|^2\,\d x
   \\
   & \qquad
   +\frac{3^{p-2}}{2\rho}\bigg(\int_{\Omega^{(1)}_\lambda}
   |\nabla u|^p\psi_\ep^p\,\d x\bigg)^{\frac{p-2}{p}}\,\bigg(
   \int_{\Omega^{(1)}_\lambda}|\nabla \psi_\ep|^p\,(w_\lambda^+)^p\bigg)^{\frac{2}{p}}
   \\[0.1cm]
   & (\text{since $0\leq \psi_\ep\leq 1$ and 
   $0\leq w_\lambda^+\leq u\leq \|u\|_{L^\infty(\Omega_\lambda)}$}) 
   \\
   & \leq 
   \frac{\rho}{2}\int_{\Omega^{(1)}_\lambda}
   \psi_\ep^{p+q}\,(|\nabla u|+|\nabla u_\lambda|)^{p-2}|\nabla w_\lambda^+|^2\,\d x
   \\
   & \qquad 
   + \frac{\mathbf{c}}{\rho}
   \bigg(\int_{\Omega^{(1)}_\lambda}
   |\nabla u|^p\,\d x\bigg)^{\frac{p-2}{p}}\,\bigg(
   \int_{\Omega^{(1)}_\lambda}|\nabla \psi_\ep|^p\bigg)^{\frac{2}{p}}
   \\[0.1cm]
   & (\text{reminding that $a\geq 0$ and $p,q\geq 2$}) \\
   & \leq 
   \frac{\rho}{2}\int_{\Omega_\lambda}\psi_\ep^{p+q}\,\big\{
   p(|\nabla u|+|\nabla u_\lambda|)^{p-2}
   + q a(x)\,(|\nabla u|+|\nabla u_\lambda|)^{q-2}\big\}
   \cdot |\nabla w_\lambda^+|^2\,\d x \\
   & \qquad
   + \frac{\mathbf{c}}{\rho}
   \bigg(\int_{\Omega_\lambda}
   |\nabla u|^p\,\d x\bigg)^{\frac{p-2}{p}}\,\bigg(
   \int_{\Omega_\lambda}|\nabla \psi_\ep|^p\bigg)^{\frac{2}{p}},
  \end{align*}
  \vspace*{0.05cm}
  where $\mathbf{c} = \mathbf{c}(p,\lambda,\|u\|_{L^\infty(\Omega_\lambda)}) > 0$. 
  Summing up, we have obtained the estimate
  \begin{equation} \label{eq.estimI1ptouse}
   \begin{split}
   I_{p,1} & \leq \frac{\rho}{2}\int_{\Omega_\lambda}\psi_\ep^{p+q}\,\big\{
   p(|\nabla u|+|\nabla u_\lambda|)^{p-2}
   + q a(x)\,(|\nabla u|+|\nabla u_\lambda|)^{q-2}\big\}
   |\nabla w_\lambda^+|^2\,\d x \\
   & \qquad
   + \frac{\mathbf{c}}{\rho}
   \bigg(\int_{\Omega_\lambda}
   |\nabla u|^p\,\d x\bigg)^{\frac{p-2}{p}}\,\bigg(
   \int_{\Omega_\lambda}|\nabla \psi_\ep|^p\bigg)^{\frac{2}{p}},
   \end{split}
  \end{equation}
  holding true for every choice of $\rho > 0$.
  \medskip
  
  \textsc{Step II: Estimate of $I_{p,2}$}.
  By definition, for every $x\in\Omega_\lambda^{(2)}$ we have
  \begin{equation} \label{eq.estimInOmegasec}
  \frac{1}{2}|\nabla u_\lambda| \leq 
   |\nabla u_\lambda|-|\nabla u|
   \leq |\nabla w_\lambda| \leq |\nabla u_\lambda|+|\nabla u|
   \leq \frac{3}{2}|\nabla u_\lambda|.
  \end{equation}
  Using again the weighted
  Young i\-ne\-qua\-li\-ty and \eqref{eq.estimInOmegasec}, for every $\rho > 0$ we get
  \begin{align*}
   I_{p,2} & \leq
   \Big(1-\frac{1}{p}\Big)\rho^{\frac{p}{p-1}}\int_{\Omega^{(2)}_\lambda}
   (|\nabla u|+|\nabla u_\lambda|)^{\frac{p(p-2)}{p-1}}|\nabla w_\lambda^+|^{\frac{p}{p-1}}
   \,\psi_\ep^{\frac{p(p+q-1)}{p-1}}\,\d x
   \\
   & \qquad +\frac{1}{p\rho^p}
   \int_{\Omega^{(2)}_\lambda}
   |\nabla \psi_\ep|^p\,
   (w_\lambda^+)^p\,\d x \\
   & = \Big(1-\frac{1}{p}\Big)
   \rho^{\frac{p}{p-1}}\int_{\Omega^{(2)}_\lambda}
   (|\nabla u|+|\nabla u_\lambda|)^{\frac{p(p-2)}{p-1}}|\nabla w_\lambda^+|^{\frac{p}{p-1}-2}\,
   |\nabla w_\lambda^+|^2
   \,\psi_\ep^{p+q+\frac{q}{p-1}}\,\d x
   \\
   & \qquad+\frac{1}{p\rho^p}
   \int_{\Omega^{(2)}_\lambda}
   |\nabla \psi_\ep|^p\,
   (w_\lambda^+)^p\,\d x \\
   & (\text{remind that $p\geq 2$, so that $p/(p-1)-2\leq 0$}) \\
   & 
   \leq c_p\,
   \rho^{\frac{p}{p-1}}\int_{\Omega^{(2)}_\lambda}
   |\nabla u_\lambda|^{\frac{p(p-2)}{p-1}}|\nabla u_\lambda|^{\frac{p}{p-1}-2}\,
   |\nabla w_\lambda^+|^2
   \,\psi_\ep^{p+q+\frac{q}{p-1}}\,\d x
   \\
   & \qquad+\frac{1}{p\rho^p}
   \int_{\Omega^{(2)}_\lambda}
   |\nabla \psi_\ep|^p\,
   (w_\lambda^+)^p\,\d x = (\bullet),
  \end{align*}
  where we have used the notation
  $$c_p := \Big(1-\frac{1}{p}\Big)
   \Big(\frac{3}{2}\Big)^{\frac{p(p-2)}{p-1}}
   \Big(\frac{1}{2}\Big)^{\frac{p}{p-1}-2};$$
  from this, reminding that $0\leq\psi_\ep\leq 1$ and
  $0\leq w_\lambda^+\leq \|u\|_{L^\infty(\Omega_\lambda)}$, we have
  \begin{align*}
   (\bullet) & = c_p\,\rho^{\frac{p}{p-1}}\int_{\Omega^{(2)}_\lambda}
   |\nabla u_\lambda|^{p-2}\,
   |\nabla w_\lambda^+|^2
   \,\psi_\ep^{p+q}\cdot\psi_\ep^{\frac{q}{p-1}}\,\d x
   \\
   & \qquad+\frac{1}{p\rho^p}
   \int_{\Omega^{(2)}_\lambda}
   |\nabla \psi_\ep|^p\,
   (w_\lambda^+)^p\,\d x \\
   & \leq \mathbf{c}'\,\rho^{\frac{p}{p-1}}
   \int_{\Omega^{(2)}_\lambda}
   |\nabla u_\lambda|^{p-2}\,
   |\nabla w_\lambda^+|^2
   \,\psi_\ep^{p+q}\,\d x
   +\frac{\mathbf{c}'}{\rho^p}
   \int_{\Omega^{(2)}_\lambda}
   |\nabla \psi_\ep|^p\,\d x \\
   & (\text{since $p\geq 2$ and the function $a$ is non-negative}) \\
   & \leq \mathbf{c}'\,
   \rho^{\frac{p}{p-1}}
   \int_{\Omega_\lambda}\psi_\ep^{p+q}\,\big\{
   p(|\nabla u|+|\nabla u_\lambda|)^{p-2}
   + q a(x)\,(|\nabla u|+|\nabla u_\lambda|)^{q-2}\big\}
   \cdot |\nabla w_\lambda^+|^2\,\d x
   \\
   & \qquad+\frac{\mathbf{c}'}{\rho^p}
   \int_{\Omega_\lambda}
   |\nabla \psi_\ep|^p\,\d x,
  \end{align*}
  \vspace*{0.05cm}
  where $\mathbf{c}' = \mathbf{c}'(p,\lambda,\|u\|_{L^\infty(\Omega_\lambda)}) > 0$.
  Summing up, we have obtained the estimate
  \begin{equation} \label{eq.estimI2ptouse}
   \begin{split}
   I_{p,2} & \leq 
   \mathbf{c}'\,
   \rho^{\frac{p}{p-1}}
   \int_{\Omega_\lambda}\psi_\ep^{p+q}\,\big\{
   p(|\nabla u|+|\nabla u_\lambda|)^{p-2}
   + q a(x)\,(|\nabla u|+|\nabla u_\lambda|)^{q-2}\big\}
   |\nabla w_\lambda^+|^2\,\d x
   \\
   & \qquad+\frac{\mathbf{c}'}{\rho^p}
   \int_{\Omega_\lambda}
   |\nabla \psi_\ep|^p\,\d x,
   \end{split}
  \end{equation}
  Gathering together \eqref{eq.estimI1ptouse} and
  \eqref{eq.estimI2ptouse}, we finally derive
  \begin{equation} \label{eq.estimFinalIp}
   \begin{split}
    I_p & \leq \Big(\frac{\rho}{2}+\kappa\rho^{\frac{p}{p-1}}\Big)\times \\
    & \qquad \times\int_{\Omega_\lambda}\psi_\ep^{p+q}\,\big\{
   p(|\nabla u|+|\nabla u_\lambda|)^{p-2}
   + q a(x)\,(|\nabla u|+|\nabla u_\lambda|)^{q-2}\big\}
   |\nabla w_\lambda^+|^2\,\d x
   \\ 
   & + \frac{\kappa}{\rho}
   \bigg(\int_{\Omega_\lambda}
   |\nabla u|^p\,\d x\bigg)^{\frac{p-2}{p}}\,\bigg(
   \int_{\Omega_\lambda}|\nabla \psi_\ep|^p\bigg)^{\frac{2}{p}}
   +\frac{\kappa}{\rho^p}
   \int_{\Omega_\lambda}
   |\nabla \psi_\ep|^p\,\d x,
   \end{split}
  \end{equation}
  for a suitable constant $\kappa$ only depending
  on $p,\lambda$ and $\|u\|_{L^\infty(\Omega_\lambda)}$.
  Furthermore, by arguing exactly in the same way, we obtain the following
  analogous
 estimate for $I_q$,
 \begin{equation} \label{eq.estimFinalIq}
   \begin{split}
    I_q & \leq \Big(\frac{\rho}{2}+\kappa'\rho^{\frac{q}{q-1}}\Big)\times \\
    & \qquad \times\int_{\Omega_\lambda}\psi_\ep^{p+q}\,\big\{
   p(|\nabla u|+|\nabla u_\lambda|)^{p-2}
   + q a(x)\,(|\nabla u|+|\nabla u_\lambda|)^{q-2}\big\}
   |\nabla w_\lambda^+|^2\,\d x
   \\ 
   & + \frac{\kappa'}{\rho}
   \bigg(\int_{\Omega_\lambda}
   |\nabla u|^q\,\d x\bigg)^{\frac{q-2}{q}}\,\bigg(
   \int_{\Omega_\lambda}|\nabla \psi_\ep|^q\bigg)^{\frac{2}{q}}
   +\frac{\kappa'}{\rho^q}
   \int_{\Omega_\lambda}
   |\nabla \psi_\ep|^q\,\d x,
   \end{split}
  \end{equation}
  where $\kappa'$ is another constant only depending on 
  $q,\lambda$ and $\|u\|_{L^\infty(\Omega_\lambda)}$.
  \medskip
  
  With the above estimates for $I_p$ and $I_q$ at hand, we are finally ready
  to conclude the proof: in fact, by combining 
   \eqref{eq.tostartfrom}, 
   \eqref{eq.estimFinalIp} and \eqref{eq.estimFinalIq}, we get
   \begin{equation} \label{eq.finalFatou}
   \begin{split}
   & \Big(C_1-C_0\rho-C_0\kappa\,\rho^{\frac{p}{p-1}}-C_0\kappa'\rho^{\frac{q}{q-1}}\Big)
    \times \\
    &\quad \times \int_{\Omega_\lambda}\psi_\ep^{p+q}\,\big\{
   p(|\nabla u|+|\nabla u_\lambda|)^{p-2}
   + q a(x)\,(|\nabla u|+|\nabla u_\lambda|)^{q-2}\big\}
   |\nabla w_\lambda^+|^2\,\d x
   \\
   & \qquad\quad
   \leq 
   C_0\,\bigg\{
   \frac{\kappa}{\rho}
   \bigg(\int_{\Omega_\lambda}
   |\nabla u|^p\,\d x\bigg)^{\frac{p-2}{p}}\,\bigg(
   \int_{\Omega_\lambda}|\nabla \psi_\ep|^p\bigg)^{\frac{2}{p}}
   +\frac{\kappa}{\rho^p}
   \int_{\Omega_\lambda}
   |\nabla \psi_\ep|^p\,\d x
   \\[0.1cm]
   & \qquad\qquad\qquad\quad
   + \frac{\kappa'}{\rho}
   \bigg(\int_{\Omega_\lambda}
   |\nabla u|^q\,\d x\bigg)^{\frac{q-2}{q}}\,\bigg(
   \int_{\Omega_\lambda}|\nabla \psi_\ep|^q\bigg)^{\frac{2}{q}}
   +\frac{\kappa}{\rho^q}
   \int_{\Omega_\lambda}
   |\nabla \psi_\ep|^q\,\d x \\[0.1cm]
   &\qquad\qquad\qquad\qquad\qquad\quad
   + \int_{\Omega_\lambda}
   (w_\lambda^+)^2\,\psi_\ep^{p+q}\,\d x\bigg\}
   \\
   & \qquad\quad (\text{since $0\leq\psi_\ep\leq 1$ and
  $0\leq w_\lambda^+\leq \|u\|_{L^\infty(\Omega_\lambda)}$}) \\
  & \qquad\quad
   \leq 
   C_0\,\bigg\{
   \frac{\kappa}{\rho}
   \bigg(\int_{\Omega_\lambda}
   |\nabla u|^p\,\d x\bigg)^{\frac{p-2}{p}}\,\bigg(
   \int_{\Omega_\lambda}|\nabla \psi_\ep|^p\bigg)^{\frac{2}{p}}
   +\frac{\kappa}{\rho^p}
   \int_{\Omega_\lambda}
   |\nabla \psi_\ep|^p\,\d x
   \\[0.1cm]
   & \qquad\qquad\qquad\quad
   + \frac{\kappa'}{\rho}
   \bigg(\int_{\Omega_\lambda}
   |\nabla u|^q\,\d x\bigg)^{\frac{q-2}{q}}\,\bigg(
   \int_{\Omega_\lambda}|\nabla \psi_\ep|^q\bigg)^{\frac{2}{q}}
   +\frac{\kappa}{\rho^q}
   \int_{\Omega_\lambda}
   |\nabla \psi_\ep|^q\,\d x \\[0.1cm]
   &\qquad\qquad\qquad\qquad\qquad\quad
   + \|u\|^2_{L^\infty(\Omega_\lambda)}\cdot\mathcal{H}^N(\Omega_\lambda)\bigg\};
   \end{split}
  \end{equation}
  from this, by choosing $\rho > 0$ in such a way that
  $$C_1-C_0\rho-C_0\kappa\,\rho^{\frac{p}{p-1}}-C_0\kappa'\rho^{\frac{q}{q-1}} < \frac{1}{2},$$
  and by letting $\varepsilon\to 0$ with the aid of Fatou's lemma
  (remind the properties (1)-to-(4) of the function
  $\psi_\ep$ and that
  $u\in C^1(\overline{\Omega}_\lambda)$ if $\lambda < 0$), we obtain
  $$\int_{\Omega_\lambda} \big[ p(|\nabla u|+|\nabla u_\lambda|)^{p-2}
   + q a(x)\,(|\nabla u|+|\nabla u_\lambda|)^{q-2}\big]
   |\nabla w_\lambda^+|^2\,\d x
   \leq \mathbf{c}_0,$$
   where $\mathbf{c}_0 = 
   2\,C_0\,\|u\|^2_{L^\infty(\Omega_\lambda)}\cdot\mathcal{H}^N(\Omega_\lambda)$.
   This is ends the proof.
 \end{proof}
 Another key tool for the proof of Theorem \ref{thm:symmetry} is Lemma
 \ref{conncomp} below. Before stating this result, we first introduce
 a notation: for every fixed
 $\lambda\in(\mathbf{a},0)$, we define 
 \begin{equation} \label{eq.defZlambda}
  \mathcal{Z}_\lambda := \big\{x\in\Omega_\lambda\setminus R_\lambda(\Gamma):\,
  \nabla u(x) = \nabla u_\lambda(x) = 0\big\}.
 \end{equation}
 We also notice that, 
 since $u,u_\lambda\in C^1(\overline{\Omega}_\lambda\setminus R_\lambda(\Gamma))$,
 the set $\mathcal{Z}_\lambda$ is closed (in $\Omega_\lambda$).
 \begin{lem}\label{conncomp}
  Let $\lambda\in (\mathbf{a},0)$ and let $\mathcal{C}_\lambda \subseteq \Omega_\lambda \setminus
	(R_\lambda(\Gamma) \cup \mathcal{Z}_\lambda)$ be a
	\emph{connected component} of \emph{(}the open set\emph{)}
	$\Omega_\lambda \setminus
	(R_\lambda(\Gamma) \cup \mathcal{Z}_\lambda)$. 	
	If 
	$u \equiv u_\lambda$ in $\mathcal{C}_\lambda$, then 
	$$\mathcal{C}_\lambda = \emptyset.$$
 \end{lem}
 \begin{proof}
  We first notice that, since it is a connected component, the set
 $\mathcal{C}_\lambda$ is surely open. In order to prove the lemma, we then
 consider the $\Pi_\lambda$-symmetric set
 $$\mathcal{C} := \mathcal{C}_\lambda\cup R_\lambda(\mathcal{C}_\lambda),$$
 and we show that $\mathcal{C} = \emptyset$. To this end, we argue by contradiction
 and we assume that
 $$\mathcal{C} \neq \emptyset.$$
 Since $\mathcal{C}_\lambda$ is open and $R_\lambda$ is a bijective linear map,
 also the set $\mathcal{C}$ is open; as a consequence, since
 $u > 0$ on $\Omega\setminus\Gamma$ and $f$
 is (continuous and) positive on $(0,\infty)$, we have
 \begin{equation} \label{eq.IntegralPositivo}
  \int_{\mathcal{C}}f(u)\,\d x > 0.
 \end{equation}
 On the other hand, since $u$ is a solution of \eqref{eq:Riey2}, one has
 \begin{equation} \label{eq.wherechoosetest}
  \begin{split}
  0 \leq \int_{\Omega}f(u)\varphi\,\d x =
  \int_\Omega \big(p|\nabla u|^{p-2}+qa(x)|\nabla u|^{q-2}\big)
	\langle \nabla u, \nabla
	\varphi\rangle\,\d x
  \end{split}
 \end{equation}
 for every function $\varphi\in C^1_c(\Omega\setminus\Gamma)$ such that
 $\varphi\geq 0$ on $\Omega\setminus\Gamma$.
 We now aim at choosing an \emph{ad-hoc} test function in 
 \eqref{eq.wherechoosetest} allowing us to contradict \eqref{eq.IntegralPositivo}.
 \vspace*{0.05cm}
 
 To begin with, since $\Gamma_0 := \Gamma\cup R_\lambda(\Gamma)$ has vanishing
 $q$-capacity (as the same is true of both $\Gamma$ and $R_\lambda(\Gamma)$),
 we can imitate the construction of the family $\{\psi_\ep\}$ in \eqref{test1}: 
 this leads
 to another family of functions, say $\{\gamma_\ep\}$, satisfying the following properties:
 \begin{enumerate}
  \item $\gamma_\ep$ is Lipschitz-continuous in $\R^N$, so that
  $\gamma_\varepsilon\in W^{1,\infty}(\R^N)$;
  
  \item $0\leq \gamma_\varepsilon\leq 1$ on $\R^N$;
  
  \item $\gamma_\ep\equiv 1$ on $\R^N\setminus \mathcal{O}^\lambda_\ep$
  and $\gamma_\ep \equiv 0$ on $\mathcal{W}^\lambda_\ep$, where
  $$\mathcal{O}^\lambda_\ep := \big\{x\in\R^N:\,
  \mathrm{dist}(x,\Gamma_0) < \ep\big\},$$ 
  and $\mathcal{W}^\lambda_\ep\subseteq \mathcal{O}^\lambda_\ep$ is a suitable
  neighborhood of $\Gamma_0$;
  
  \item there exists a constant $C > 0$, independent of $\varepsilon$, such that
  \begin{equation} \label{eq.estimNablagamma}
   \int_{\R^N}|\nabla
  \gamma_{\varepsilon}|^q\,\d x \leq C \varepsilon.
  \end{equation}
  \end{enumerate} 
  Moreover, for every fixed $\ep > 0$ we consider the maps
  $G_\ep, h_\ep:[0,\infty)\to\R$ defined as
  $$G_\ep(t) := \begin{cases}
  0, & \text{if $0\leq t \leq \ep$}, \\
  2t-{2\ep}, & \text{if $\ep < t \leq 2\ep$}, \\
  t, &\text{if $t > 2\ep$}.
  \end{cases}\qquad\text{and}\qquad h_\ep(t) := \frac{G_\ep(t)}{t}.$$
  Using the family $\{\gamma_\ep\}$ and the function $h_\ep$ just introduced,
  we set
  $$\varphi_\ep:\R^N\to\R, \qquad \varphi_\ep(x) := \begin{cases}
  h_\ep(|\nabla u(x)|)\cdot\gamma^2_\ep(x), & \text{if $x\in \mathcal{C}$}, \\
  0, &\text{if $x\notin\mathcal{C}$},
  \end{cases}$$
  and we claim that $\varphi_\ep$ can be chosen as a test function in 
  \eqref{eq.wherechoosetest}.
  In fact, since both $\gamma_\ep$ and $h_\ep$ are Lipschitz-continuous on the whole of
   $\R^N$, we clearly have 
   $\varphi_\ep\in\mathrm{Lip}(\mathcal{C})$;
   mo\-re\-o\-ver, since $u\equiv u_\lambda$ in $\mathcal{C}_\lambda$ and $u\equiv 0$
   on $\partial\Omega$, one has
   \begin{equation} \label{eq.uequivzeroinside}
    \text{$u\equiv 0$ on $K := R_\lambda(\partial\mathcal{C}_\lambda\cap\partial\Omega)$}.
   \end{equation}
   Now, reminding that $u > 0$ in $\Omega\setminus\Gamma$ 
   and $u\in C^1(\overline{\Omega}\setminus\Gamma)$, 
   from 
   \eqref{eq.uequivzeroinside} we infer that
   $$\text{$\nabla u = \nabla u_\lambda = 0$ on $\partial\mathcal{C}\setminus \Gamma_0$};$$
   as a consequence, taking into account the very definition of $h_\ep$, 
   it is not difficult to recognize that
   $\varphi_\ep\in\mathrm{Lip}(\R^N)$ and that
   $\mathrm{supp}(\varphi_\ep)\subseteq\mathcal{C}
   \setminus\Gamma_0$.
   By a standard density argument we are then entitled to use
   $\varphi_\ep$ as a test function in \eqref{eq.wherechoosetest}, obtaining 
   \begin{equation} \nonumber
	\begin{split}
	0 \leq & \int_{\mathcal{C}} f(u)h_\varepsilon(|\nabla u|)\gamma^2_\varepsilon\,\d x  
	= \int_{\mathcal{C}}\gamma_\varepsilon^2\cdot
	\big(p|\nabla u|^{p-2}+a(x)|\nabla u|^{q-2}\big)\,\langle 
	\nabla u,
	\nabla(h_\ep\circ |\nabla u|)\rangle\,\d x\\
	& \qquad
	+ 2\int_{\mathcal{C}} 
	h_{\varepsilon}(|\nabla u|)\,\gamma_\varepsilon\cdot\big(
	p|\nabla u|^{p-2}+qa(x)|\nabla u|^{q-2} \big)\,\langle \nabla u, \nabla \gamma_\varepsilon
	\rangle\,\d x.
	\end{split}
	\end{equation}
	To proceed further we observe that,
	since $u\in C^1(\overline{\Omega}_\lambda)$
	and $u\equiv u_\lambda$ on $\mathcal{C}_\lambda\subseteq\Omega_\lambda$, we have
	$u\in C^1(\overline{\mathcal{C}})$. We can then invoke the regularity
	results proved by Riey \cite{riey}, ensuring that
	\begin{equation} \label{eq.regulriey}
	 u\in W^{2,s}(\mathcal{C})\quad \text{for a suitable $s = s_p\in (1,2]$}.
	\end{equation}
	As a consequence, we can write
	\begin{equation} \label{eq.IntegraliwithDsec}
	 \begin{split}
	  & \int_{\mathcal{C}}\gamma_\varepsilon^2\cdot
	\big(p|\nabla u|^{p-2}+a(x)|\nabla u|^{q-2}\big)\,\langle 
	\nabla u,
	\nabla(h_\ep\circ |\nabla u|)\rangle\,\d x\\
	& \qquad = \int_{\mathcal{C}}h_{\varepsilon}'(|\nabla u|)\,\gamma_\varepsilon^2\cdot
	\big(p|\nabla u|^{p-2}+a(x)|\nabla u|^{q-2}\big)\,\langle 
	\nabla u,
	\nabla |\nabla u|\rangle\,\d x.
	 \end{split}
	\end{equation}
	From this, since
	 by definition one has
	\begin{equation} \label{eq.estimhhfirst}
	 h_\varepsilon(t)\leq 1\quad \text{and} \quad h'_\varepsilon(t)\leq 2/\varepsilon,
	 \end{equation}
	using Schwartz's inequality, 
	\eqref{eq.estimhhfirst} and reminding
	that $0\leq\gamma_\ep\leq 1$, we then get
	{\allowdisplaybreaks
	\begin{equation} \label{eq.topasstolimitLebesguelemmaCC}
	\begin{split}
	 & 0 \leq \int_{\mathcal{C}}f(u)h_\varepsilon(|\nabla u|)\gamma^2_\varepsilon\,\d x 
	 \\
	& \leq 2\int_{\mathcal{C}\cap\{\varepsilon <|\nabla u|< 2\varepsilon\}}
	\frac{|\nabla u|}{\varepsilon}\cdot
	\big(p|\nabla u|^{p-2}+qa(x)|\nabla u|^{q-2}\big)\,\|D^2 u\|\,\gamma_\varepsilon^2\,\d x
	\\
	&\qquad\quad 
	+2 \int_{\mathcal{C}}\psi_\ep\cdot
	\big(p|\nabla u|^{p-1} + qa(x)|\nabla u|^{q-1}\big)\,|\nabla \gamma_\varepsilon|\,\d x
	\\
	&  
	\leq 4p\int_{\mathcal{C} \cap\{\varepsilon <|\nabla u| < 2\varepsilon\}}
	|\nabla u|^{p-2} \|D^2 u\| \gamma_\varepsilon^2\,\d x + 
	4q \int_{\mathcal{C} \cap\{\varepsilon <|\nabla u| < 2\varepsilon\}}
	a(x)|\nabla u|^{q-2} \|D^2 u\| \gamma_\varepsilon^2\,\d x \\
	&\qquad\quad +2 \int_{\mathcal{C}}\psi_\ep\cdot
	\big(p|\nabla u|^{p-1} + qa(x)|\nabla u|^{q-1}\big)\,|\nabla \gamma_\varepsilon|\,\d x \\
	& \leq 4p\int_{\mathcal{C}} |\nabla	u|^{p-2} \|D^2 u\| \gamma_\varepsilon^2\cdot
	\mathbf{1}_{\mathcal{D}_\varepsilon}\,\,d x + 
	4q\int_{\mathcal{C}}a(x) |\nabla	u|^{q-2} \|D^2 u\| \gamma_\varepsilon^2 \cdot
	\mathbf{1}_{\mathcal{D}_\varepsilon}\,\d x  \\
	&\qquad\quad + 2p
	\bigg(\int_{\mathcal{C}} |\nabla u|^{p}\,\d x \bigg)^{\frac{p-1}{p}} 
	\bigg(\int_{\mathcal{C}}|\nabla \gamma_\varepsilon|^p\,\d x\bigg)^{\frac{1}{p}} 
	\\
	& \qquad\qquad\quad 
	+ 2q\,\|a\|_{L^\infty(\Omega)}
	\bigg(\int_{\mathcal{C}} |\nabla u|^{q}\,\d x \bigg)^{\frac{q-1}{q}}
	\bigg(\int_{\mathcal{C}}|\nabla \gamma_\varepsilon|^q\,\d x\bigg)^{\frac{1}{q}},
	\end{split}
   \end{equation}
   }
	where $\mathcal{D}_\ep := \mathcal{C}\cap \{\ep<|\nabla u| < 2\ep\}$
	and $\mathbf{1}_{\mathcal{D}_\ep}$ is the indicator function
	of $\mathcal{D}_\ep$.
	\vspace*{0.05cm}
	
	We now aim to apply a dominated-convergence argument to let $\ep\to 0^+$
	in \eqref{eq.topasstolimitLebesguelemmaCC}.
	To this end we first notice that, by definition of $\mathcal{D}_\ep$, we have
	(a.e.\,on $\mathcal{C}$)
	$$\lim_{\ep\to 0^+}\big(|\nabla	u|^{p-2} \|D^2 u\| \gamma_\varepsilon^2\cdot
	\mathbf{1}_{\mathcal{D}_\varepsilon}\big)
	= \lim_{\ep\to 0^+}
	\big(a(\cdot) |\nabla	u|^{q-2} \|D^2 u\| \gamma_\varepsilon^2 \cdot
	\mathbf{1}_{\mathcal{D}_\varepsilon}\big) = 0;$$
	moreover, reminding that $u\in C^1(\overline{\mathcal{C}})$
	and $q > p$, one has
	\begin{equation*}
	\begin{split}
	& \big||\nabla	u|^{p-2} \|D^2 u\| \gamma_\varepsilon^2\cdot
	\mathbf{1}_{\mathcal{D}_\varepsilon}\big|\leq
	|\nabla	u|^{p-2} \|D^2 u\|
	\qquad\text{and} \\[0.15cm]
	&  
	\big|a(x)|\nabla	u|^{q-2} \|D^2 u\| \gamma_\varepsilon^2\cdot
	\mathbf{1}_{\mathcal{D}_\varepsilon}\big|\leq
	C\,|\nabla	u|^{p-2} \|D^2 u\|,
	\end{split}
	\end{equation*}
	where $C := \|a\|_{L^\infty(\Omega)}\cdot
	\|\nabla u\|^{q-p-2}_{L^\infty(\overline{\mathcal{C}})}$.
	Appealing once again to some results by Riey \cite{riey} (see, precisely,
	Corollary 1 with $\beta = \gamma = 0$), we know that
	$$|\nabla	u|^{p-2} \|D^2 u\|\in L^1(\mathcal{C});$$
	as a consequence, by taking into account \eqref{eq.estimhhfirst} 
	and the properties of $\gamma_\ep$, we can pass to the limit
	as $\ep\to 0^+$ in \eqref{eq.topasstolimitLebesguelemmaCC} with the aid
	of Lebesgue's theorem,
	obtaining
	$$\int_{\mathcal{C}}f(u)\,\d x = 0.$$
	This is clearly in contradiction with \eqref{eq.IntegralPositivo}, and the proof is complete.
 \end{proof}
 
\section{Proof of Theorem \ref{thm:symmetry}} \label{sec.mainresult1}

\begin{proof}[Proof of Theorem \ref{thm:symmetry}] 
By assumptions, the singular set $\Gamma$ is contained in the hyperplane $\{x_1 = 0\}$, then the moving plane procedure can be started in the standard way, see e.g \cite{EMS} for the $p$-laplacian case, by using the weak comparison principle in small domains, see \cite[Theorem 4.3]{riey}.
	Indeed, for $\mathbf{a}
	< \lambda < \mathbf{a} + \tau$ with $\tau>0$ small enough, the singularity does not play any role. Therefore, recalling that $w_\lambda$ has a
	singularity at $\Gamma$ and at $R_\lambda (\Gamma)$, we have that
	$w_\lambda \leq 0$ in $\Omega_\lambda$. To proceed
	further we define
	\begin{equation}\nonumber
	\Lambda_0=\{\mathbf{a}<\lambda<0 : u\leq
	u_{t}\,\,\,\text{in}\,\,\,\Omega_t\setminus
	R_t(\Gamma)\,\,\,\text{for all $t\in(\mathbf{a},\lambda]$}\}
	\end{equation}
	and $\lambda_0 = \sup \Lambda_0$, since we proved above that $\Lambda_0$ is not empty. To prove our 
	result we have to
	show that $\lambda_0 = 0$.
	To do this we
	assume that $\lambda_0 < 0$ and we reach a contradiction by proving
	that $u \leq u_{\lambda_0 + \tau}$ in $\Omega_{\lambda_0 + \tau}
	\setminus R_{\lambda_0 + \tau} (\Gamma)$ for any $0 < \tau <
	\bar{\tau}$ for some small $\bar{\tau}>0$. We  remark that
	$|\mathcal{Z}_{\lambda_0}|=0$, see \cite{DS1, riey}.
	Let us take $\mathcal{H}_{\lambda_0}\subset \Omega_{\lambda_0}$ be an open set such that $\mathcal 
	Z_{\lambda_0}\cap\Omega_{\lambda_0}\subset
	\mathcal{H}_{\lambda_0} \subset \subset \Omega$. We note that the existence of such a set is guaranteed 
	by Theorem \ref{thm.Hopftype}. Moreover note  that, since $|\mathcal Z_{\lambda_0}|=0$, we can take $
	\mathcal{H}_{\lambda_0}$ of arbitrarily small measure. By
	continuity we know that $u \leq u_{\lambda_0}$ in $\Omega_{\lambda_0}
	\setminus R_{\lambda_0} (\Gamma)$.
	We can exploit Theorem \ref{thm.SMP}
	to get that, in any connected component of 
	$\Omega_{\lambda_{0}}\setminus \mathcal Z_{\lambda_0}$, we have
	$$
	u<u_{\lambda_0} \qquad\text{or}\qquad u\equiv u_{\lambda_0}.$$
	The case $u\equiv u_{\lambda_0}$ in some  connected component
	$\mathcal{C}_{\lambda_{0}}$ of $\Omega_{\lambda_{0}}\setminus \mathcal Z_{\lambda_0}$ is not
	possible, since by symmetry, it would imply the existence of  a local symmetry phenomenon and 
	consequently that $\Omega \setminus \mathcal  Z_{\lambda_0}$ would be not connected,  in  spite of what 
	we proved in Lemma \ref{conncomp}. Hence we deduce that $u <
	u_{\lambda_0}$ in $\Omega_{\lambda_0} \setminus R_{\lambda_0}
	(\Gamma)$. Therefore, given a compact set $\mathcal{K} \subset
	\Omega_{\lambda_0} \setminus (R_{\lambda_0} (\Gamma)\cup \mathcal{H}_{\lambda_0})$, by
	uniform continuity we can ensure that $u < u_{\lambda_0+\tau}$ in
	$\mathcal{K}$ for any $0 < \tau < \bar{\tau}$ for some small $\bar{\tau}>0$.
	Note that to do this we implicitly assume, with no loss of
	generality, that $R_{\lambda_0} (\Gamma)$ remains bounded away
	from $\mathcal{K}$. 
	
	Arguing in a similar fashion as in Lemma \ref{leaiuto}, we
	consider
	\begin{equation}\label{eq:moduloooooooo}
	\varphi_\varepsilon := w^+_{\lambda_0 + \tau} 
	\psi_\varepsilon^{p+q}\cdot\mathbf{1}_{\Omega_{\lambda_0 + \tau}}
	= \begin{cases}
	w^+_{\lambda_0 + \tau} 
	\psi_\varepsilon^{p+q}, & \text{in $\Omega_{\lambda_0+\tau}$}, \\
	0, & \text{otherwise}.
	\end{cases}
	\end{equation}
	By density arguments as above, we plug $\varphi_\varepsilon$ as test
	function in \eqref{debil1} and \eqref{eq.PDEulambda} so that,
	subtracting, we get
	
	\begin{equation} \label{eq:Step1}
	\begin{split} &
	p\int_{\Omega_{\lambda_0 + \tau} \setminus \mathcal{K}} 
	\langle|\nabla u|^{p-2} \nabla u - |\nabla u_{\lambda_0 + \tau}|^{p-2}\nabla u_{\lambda_0 + \tau},
	\nabla w^+_{\lambda_0 + \tau}\rangle \, \psi_\varepsilon^{p+q} \, \d x\\
	&+q\int_{\Omega_{\lambda_0 + \tau} \setminus \mathcal{K}} a(x)
	\langle|\nabla u|^{q-2} \nabla u
	- |\nabla u_{\lambda_0 + \tau}|^{q-2} \nabla u_{\lambda_0 + \tau},
	\nabla w^+_{\lambda_0 + \tau}\rangle \, \psi_\varepsilon^{p+q} \, \d x\\
	&+ p(p+q) \int_{\Omega_{\lambda_0 + \tau} \setminus \mathcal{K}} \langle|\nabla u|^{p-2}
	\nabla u - |\nabla u_{\lambda_0 + \tau}|^{p-2} \nabla u_{\lambda_0 +
		\tau}, \nabla \psi_\varepsilon\rangle \, 
		\psi_\varepsilon^{p+q-1} w_{\lambda_0 + \tau}^+ \, \d x \\
	&+ q(p+q) \int_{\Omega_{\lambda_0 + \tau} \setminus \mathcal{K}}a(x) \langle|\nabla u|^{q-2}
	\nabla u - |\nabla u_{\lambda_0 + \tau}|^{q-2} \nabla u_{\lambda_0 +
		\tau}, \nabla \psi_\varepsilon\rangle \,	
		\psi_\varepsilon^{p+q-1} w_{\lambda_0 + \tau}^+ \, \d x \\
	&= \int_{\Omega_{\lambda_0 + \tau} \setminus \mathcal{K}} (f(u)-f(u_\lambda))
	w_{\lambda_0 + \tau}^+ \psi_\varepsilon^{p+q} \, \d x.
	\end{split}
	\end{equation}
	Now we split the set $\Omega_{\lambda_0 + \tau}
	\setminus \mathcal{K}$ as the union of two disjoint subsets $\Omega^{(1)}_{\lambda_0 +
	\tau}$ and $\Omega^{(2)}_{\lambda_0 + \tau}$ such that $\Omega_{\lambda_0 +\tau} \setminus \mathcal{K}= \Omega^{(1)}_{\lambda_0 + \tau} \cup \Omega^{(2)}_{\lambda_0 +	\tau}$.
	In particular, we set
	\begin{equation}\nonumber
	\begin{split}
	 \Omega^{(1)}_{\lambda_0 +	\tau} &:= \{ x \in \Omega_{\lambda_0 + \tau} \setminus
	\mathcal{K} \ : \ |
	\nabla u_{\lambda_0 + \tau} (x)| < 2| \nabla u (x)|  \}\quad
	\text{and}\\ \\
	\Omega^{(2)}_{\lambda_0 +	\tau} &:= \{ x \in \Omega_{\lambda_0 + \tau} \setminus
	\mathcal{K} \ : \ | \nabla u_{\lambda_0 + \tau} (x)| \geq  2| \nabla u
	(x)|\}.
	\end{split}
	\end{equation}
	From \eqref{eq:Step1} and using \eqref{eq:inequalities}, repeating verbatim arguments along the proof of Lemma \ref{leaiuto},  we have
	
\begin{equation}\nonumber
\begin{split}
&C_1\int_{\Omega_{\lambda_0+\tau} \setminus \mathcal{K}} \psi_\ep^{p+q}\,\big\{ p(|\nabla u|+|\nabla u_{\lambda_0+\tau}|)^{p-2} + q a(x)\,(|\nabla u|+|\nabla u_{\lambda_0+\tau}|)^{q-2}\big\}|\nabla w_{\lambda_0+\tau}^+|^2 \, \d x  \\[0.1cm]
\leq& C_0\Big(\rho+\kappa\rho^{\frac{p}{p-1}}+\kappa'\rho^{\frac{q}{q-1}}\Big) \times \\
& \quad \times\int_{\Omega_{\lambda_0+\tau} \setminus \mathcal{K}} \psi_\ep^{p+q}\,\big\{ p(|\nabla u|+|\nabla u_{\lambda_0+\tau}|)^{p-2} + q a(x)\,(|\nabla u|+|\nabla u_{\lambda_0+\tau}|)^{q-2}\big\}|\nabla w_{\lambda_0+\tau}^+|^2\,\d x \\[0.1cm]
&+ C_0\,\bigg\{ \frac{\kappa}{\rho} \bigg(\int_{\Omega_{\lambda_0+\tau} \setminus \mathcal{K}} |\nabla u|^p\,\d x\bigg)^{\frac{p-2}{p}}\,\bigg( \int_{\Omega_{\lambda_0+\tau} \setminus \mathcal{K}} |\nabla \psi_\ep|^p\bigg)^{\frac{2}{p}} +\frac{\kappa}{\rho^p} \int_{\Omega_{\lambda_0+\tau} \setminus \mathcal{K}} |\nabla \psi_\ep|^p\,\d x \\[0.1cm]
&\qquad \quad  + \frac{\kappa'}{\rho} \bigg(\int_{\Omega_{\lambda_0+\tau} \setminus \mathcal{K}} |\nabla u|^q\,\d x\bigg)^{\frac{q-2}{q}}\,\bigg( \int_{\Omega_{\lambda_0+\tau} \setminus \mathcal{K}} |\nabla \psi_\ep|^q\bigg)^{\frac{2}{q}} +\frac{\kappa}{\rho^q}\int_{\Omega_{\lambda_0+\tau} \setminus \mathcal{K}} |\nabla \psi_\ep|^q\,\d x \\
&\qquad \quad  + L_f \int_{\Omega_{\lambda_0 + \tau} \setminus \mathcal{K}} (w_{\lambda_0	+ \tau}^+)^2 \psi_\varepsilon^{p+q}\, \d x. \bigg\},
\end{split}
\end{equation}
where $C_0=C_0(p,q,\lambda_0, \tau, \|u\|_{L^\infty(\Omega)})$. Taking $\rho>0$ sufficiently small in such a way that
	$$C_1-C_0\rho - C_0 \kappa \rho^\frac{p}{p-1} - C_0 \kappa' \rho^\frac{q}{q-1} < \frac{1}{2},$$ 
	as we did above passing to the
	limit for $\varepsilon \rightarrow 0$ (thanks to Fatou's lemma) we obtain
	\begin{equation}\label{eq:final}
	\begin{split}
		&\int_{\Omega_{\lambda_0+\tau} \setminus \mathcal{K}} \left[p(|\nabla u| + |\nabla u_{\lambda_0+\tau}|)^{p-2} + q a(x)(|\nabla u| + |\nabla u_{\lambda_0+\tau}|)^{q-2}\right]
		|\nabla w_{\lambda_0+\tau}^+|^2 \psi_\varepsilon^{p+q} \, \d x  \\[0.1cm]
	&\leq 2 C_0L_f \int_{\Omega_{\lambda_0 + \tau} \setminus \mathcal{K}} (w_{\lambda_0
		+ \tau}^+)^2 \, \d x.
	\end{split}
	\end{equation}

	\noindent We now observe that, since $p > 2$, we have
	$$|\nabla u|^{p-2} \leq p(|\nabla u| +
	|\nabla u_{\lambda_0+\tau}|)^{p-2} \leq p(|\nabla u| +
	|\nabla u_{\lambda_0+\tau}|)^{p-2} + q a(x)(|\nabla u| +
	|\nabla u_{\lambda_0+\tau}|)^{q-2}.$$
	Setting $\rho:=|\nabla
	u|^{p-2}$, we see that $\rho$ is bounded in
	$\Omega_{\lambda_0+\tau}$, hence $\rho \in L^1(\Omega_{\lambda_0 +
		\tau})$. By applying the weighted Poincar\'{e} inequality to
	\eqref{eq:final}, see \cite[Theorem 4.2]{riey}, we deduce that
	\begin{equation}\label{eq:final3}
	\begin{split}
	\int_{\Omega_{\lambda_0 + \tau} \setminus \mathcal{K}} & \rho |\nabla
	w_{\lambda_0 + \tau}^+|^2 \, \d x \\
	& \leq \int_{\Omega_{\lambda_0 + \tau} \setminus \mathcal{K}} \left[ p(|\nabla u| + |\nabla u_{\lambda_0 + \tau}|)^{p-2} + q a(x)(|\nabla u| + |\nabla u_{\lambda_0 + \tau}|)^{q-2} \right]  |\nabla w_{\lambda_0 + \tau}^+|^2 \, \d x \\
	& \leq 2C_0 L_f \int_{\Omega_{\lambda_0 + \tau} \setminus \mathcal{K}} (w_{\lambda_0 + \tau}^+)^2 \, \d x \\
	&\leq 2C_0 L_f C_p(|\Omega_{\lambda_0 + \tau} \setminus \mathcal{K}|)
	\int_{\Omega_{\lambda_0 + \tau} \setminus \mathcal{K}} \rho |\nabla
	w_{\lambda_0 + \tau}^+|^2 \, \d x
	\end{split}
	\end{equation}
	where $C_p(\cdot)$ tends to zero if the measure of the domain tends to zero. For $\bar{\tau}$
	small and $\mathcal{K}$ large, we may assume that
	$$2C_0 L_f C_p(|\Omega_{\lambda_0 + \tau} \setminus \mathcal{K}|) < \frac{1}{2};$$
	as a consequence by \eqref{eq:final3} and Lemma
	\ref{leaiuto} we deduce that
	$$\int_{\Omega_{\lambda_0+\tau}} \rho |\nabla
	w_{\lambda_0+\tau}^+|^2 \, \d x = \int_{\Omega_{\lambda_0 + \tau}
		\setminus \mathcal{K}}  \rho |\nabla w_{\lambda_0+\tau}^+|^2 \, \d x  = 0,$$
	and this proves that $u \leq u_{\lambda_0+\tau}$ in $\Omega_{\lambda_0 +
		\tau} \setminus R_{\lambda_0 + \tau} (\Gamma)$ for any $0 < \tau <
	\bar{\tau}$ (provided $\bar{\tau}>0$ is small enough). Such a contradiction
	shows that
	$$ \lambda_0 = 0.$$
	Since the moving plane procedure can be performed in the same way
	but in the opposite direction, then this proves the desired symmetry
	result. The fact that the solution is increasing in the
	$x_1$-direction in $\{x_1 < 0\}$ is implicit in the moving plane
	procedure.	
  \end{proof}

\section{Proof of Theorem \ref{thm:noSingularSym}} \label{sec.mainresult2}
\begin{proof}[Proof of Theorem \ref{thm:noSingularSym}] 
	First of all we observe that, in the case 
	$q>p \geq 2$, Theorem \ref{thm:noSingularSym} immediately
	follows
	from Theorem \ref{thm:symmetry} by choosing
	$\Gamma = \emptyset$. 
	As a consequence, we assume from now on
	that $1<p<2\leq q$ (the case $1<p<q<2$ is very similar).
	\vspace*{0.05cm}
	
	By proceeding exactly as in the proof of Theorem \ref{thm:symmetry}, we easily
	recognize that
	$$\Lambda_0=\{\mathbf{a}<\lambda<0 : u\leq
		u_{t}\,\,\,\text{in}\,\,\,\Omega_t \,\,\,\text{for all $t\in(\mathbf{a},\lambda]$}\}
		\neq \emptyset,$$
	and thus $\lambda_0 := \sup(\Lambda_0)\in (\mathbf{a},0]$.
	Arguing by contradiction, we suppose that 
	$$\lambda_0 < 0$$ 
	and we proceed again
	as in the proof of Theorem \ref{thm:symmetry}: choosing
	an open neighborhood $\mathcal{H}_{\lambda_0}$
	of $\mathcal{Z}_{\lambda_0}$ with arbitrary small Lebesgue measure, 
	for every compact set
	$\mathcal{K} \subseteq
	\Omega_{\lambda_0} \setminus 
	\mathcal{H}_{\lambda_0}$  we are able to
	find a suitable
	$\bar{\tau} > 0$ such that
	\begin{equation} \label{eq.signinKGammavuoto}
	\text{$u < u_{\lambda_0+\tau}$ in
	$\mathcal{K}$ for any $0 < \tau < \bar{\tau}$}.
	\end{equation}
	We now fix $\tau\in (0,\bar{\tau})$ and 
	we consider the function $\varphi:\R^N\to\R$ defined as follows:
	$$\varphi := w^+_{\lambda_0 + \tau}\cdot\mathbf{1}_{\Omega_{\lambda_0 + \tau}}.$$
	Since $u\in C^1(\overline{\Omega})$, we clearly have that
	$\varphi\in \mathrm{Lip}(\R^N)$ and
	$\mathrm{supp}(\varphi)\subseteq\Omega_{\lambda_0+\tau}$.
	By a standard density argument we can use $\varphi$ as a test function
	in \eqref{debil1} and \eqref{eq.PDEulambda}, thus obtaining
	
	\begin{equation} \label{eq:noSingStep1}
		\begin{split} &p\int_{\Omega_{\lambda_0 + \tau} 
		\setminus \mathcal{K}} \langle|\nabla u|^{p-2} \nabla u - 
		|\nabla u_{\lambda_0 + \tau}|^{p-2}\nabla u_{\lambda_0 + \tau},
			\nabla w^+_{\lambda_0 + \tau}\rangle \, \d x\\
			&\qquad
			+q\int_{\Omega_{\lambda_0 + \tau} \setminus \mathcal{K}} 
			a(x)\langle|\nabla u|^{q-2} \nabla u - |\nabla u_{\lambda_0 + \tau}|^{q-2} 
			\nabla u_{\lambda_0 + \tau},
			\nabla w^+_{\lambda_0 + \tau}\rangle \,  \d x\\
			&= \int_{\Omega_{\lambda_0 + \tau} \setminus \mathcal{K}} (f(u)-f(u_\lambda))
			w_{\lambda_0 + \tau}^+ \, \d x.
		\end{split}
	\end{equation}
	Starting from \eqref{eq:noSingStep1}, we closely follow
	the proof of Lemma \ref{leaiuto} up to formula \eqref{eq.tostartfrom}:
	since in our case we formally have $\psi_\ep \equiv 1$ (and thus
	$\nabla \psi_\ep \equiv 0$), we get
	\begin{equation}\label{eq:noSingfinal1}
		\begin{split}
			 & \int_{\Omega_{\lambda_0+\tau} \setminus \mathcal{K}} 
			 \big\{p(|\nabla u| + |\nabla u_{\lambda_0+\tau}|)^{p-2} + 
			 qa(x)(|\nabla u| + |\nabla u_{\lambda_0+\tau}|)^{q-2}\big\}\cdot
			|\nabla w_{\lambda_0+\tau}^+|^2\,\d x  \\
			& \qquad \leq C\,
			\int_{\Omega_{\lambda_0+\tau}\setminus \mathcal{K}} 
			(w_{\lambda_0 + \tau}^+)^2\,\d x,
		\end{split}
	\end{equation}
	where $C = C(p,q,\|u\|_{L^\infty(\Omega)},f) > 0$
	is a suitable constant.
	To proceed further, we set
	$$\varrho :=\left(1+|\nabla u|^2 + |\nabla u_{\lambda_0+\tau}|^2 \right)^{\frac{p-2}{2}}$$
	in order to exploit the weighted Sobolev inequality from 
	\cite{Tru}. We remind that the results of \cite{Tru} do apply if
	$\varrho \in L^1(\Omega_\lambda)$ and if there exists
	some $t > N/2$ such that
	$$1/\varrho\in L^t(\Omega_\lambda).$$
	In particular, if $\mathcal{O}\subseteq\R^N$ is any
	(non-void) open set, the space $H^1_{0,\varrho} (\mathcal{O})$ 
	(see \cite{DS1,Tru}) coincides with the closure of $C^\infty_c(\mathcal{O})$ 
	with respect to the norm
	$$
	\|w\|_\varrho\,:=\,\||\nabla w|\|_{L^2(\mathcal{O},\varrho)}:=
	\bigg(\int_{\mathcal{O}}\rho\,|\nabla w|^2\,\d x \bigg)^{\frac 12},
	$$
	and there exists a suitable constant $C_S > 0$ such that
	\begin{equation} \label{eq.SobolevTrud}
	\|w\|_{L^{2^*_\varrho}(\mathcal{O})}
	\,\leq C_S \,\||\nabla w|\|_{L^2(\mathcal{O},\varrho)}
	\qquad\text{for any $w\in H^1_{0,\varrho}(\mathcal{O})$},
	\end{equation}
	where the exponent $2^{*_\varrho}$ is defined via the relation
	$$\frac{1}{2^*_\varrho}:=\frac{1}{2}\left(1+\frac{1}{t}\right)-\frac{1}{N}.$$
	We now observe that, since $u\in C^1(\overline{\Omega})$ (and $1<p < 2$), it is possible to find
	two constants $K_1,K_2 > 0$, only
	depending  on $p$ and on $\|u\|_{C^1(\overline{\Omega})}$, such that
	\begin{equation}\label{eq:weight4}
		\left( 1+|\nabla u|^2+|\nabla u_{\lambda_0+\tau}|^2
		\right)^{\frac{2-p}{2}} \leq K_1 + K_2 |\nabla
		u_{\lambda_0+\tau}|^{2-p} \qquad\text{in $\Omega_{\lambda_0+\tau}$}.
	\end{equation}
	Using \eqref{eq:weight4} and the fact both
	$u$ and 
	$u_{\lambda_0+\tau}$ are of class $C^1$
	on $\overline{\Omega}_{\lambda_0+\tau}$
	(see \eqref{eq.defOmegalambda}
	and remind that, since $\Gamma = \emptyset$, one actually has $u\in C^1(\overline{\Omega})$),
	we deduce that
	$$
	{1}/{\varrho}:=\big( 1+|\nabla u|^2+|\nabla
	u_{\lambda_0+\tau}|^2 \big)^{\frac{2-p}{2}} \in
	L^\infty(\Omega_{\lambda_0+\tau}).
	$$
	We are then entitled to use the 
	Sobolev inequality  \eqref{eq.SobolevTrud}
	in \eqref{eq:noSingfinal1}: observing that
	\begin{equation}\label{eq:weight3}
	  \begin{split}
		\big(|\nabla u|+|\nabla u_{\lambda_0+\tau}|\big)^{2-p} & \leq
		2^{\frac{2-p}{2}} \big( |\nabla u|^2+|\nabla u_{\lambda_0+\tau}|^2
		\big)^{\frac{2-p}{2}} \\
		& \leq 2^{\frac{2-p}{2}} \big( 1+|\nabla
		u|^2+|\nabla u_{\lambda_0+\tau}|^2 \big)^{\frac{2-p}{2}},
		\end{split}
	\end{equation}
	by exploiting
	\eqref{eq:weight3} and H\"older's inequality we obtain
	\begin{equation}\label{eq:Poincare}
		\begin{split}
			& \int_{\Omega_{\lambda_0 + \tau} \setminus \mathcal{K}} \varrho |\nabla
			w_{\lambda_0+\tau}^+|^2 \,\d x\\
			& \quad
			\leq 2^{\frac{2-p}{2}}
			\int_{\Omega_{\lambda_0+\tau} \setminus \mathcal{K}} 
			 \big\{p(|\nabla u| + |\nabla u_{\lambda_0+\tau}|)^{p-2} + 
			 qa(x)(|\nabla u| + |\nabla u_{\lambda_0+\tau}|)^{q-2}\big\}\cdot
			|\nabla w_{\lambda_0+\tau}^+|^2\,\d x \\
			& \quad \leq C_p\int_{\Omega_{\lambda_0 + \tau}
				\setminus \mathcal{K}}	(w_{\lambda_0+\tau}^+)^2 \,\d x
			\\
			& \quad \leq C_p\cdot 
				\mathcal{H}^N(\Omega_{\lambda_0+\tau} \setminus
			\mathcal{K})^{\frac{1}{(\frac{2}{2^*_\varrho})'}}
			\left(\int_{\Omega_{\lambda_0 + \tau} \setminus \mathcal{K}}
			(w_{\lambda_0+\tau}^+)^{2^*_\varrho} \,\d x\right)^{\frac{2}{2^*_\varrho}}
			\\
			& \quad \leq \Theta_p(|\Omega_{\lambda_0 + \tau}
			\setminus \mathcal{K}|)\int_{\Omega_{\lambda_0 + \tau} \setminus \mathcal{K}} \varrho
			|\nabla w_{\lambda_0+\tau}^+|^2\,\d x,
		\end{split}
	\end{equation}	
	where $\Theta_p(\cdot)$ tends to zero if the measure of the domain tends to zero. For 
	$\bar{\tau}$ sufficiently
	small and $\mathcal{K}$ sufficiently large, we may assume that
	$$\Theta_p(|\Omega_{\lambda_0 + \tau} \setminus \mathcal{K}|) < \frac{1}{2};$$
	as a consequence, from \eqref{eq:Poincare} and \eqref{eq.signinKGammavuoto} we deduce that
	$$\int_{\Omega_{\lambda_0+\tau}} \rho |\nabla
	w_{\lambda_0+\tau}^+|^2 \, \d x = \int_{\Omega_{\lambda_0 + \tau}
		\setminus \mathcal{K}}  \rho |\nabla w_{\lambda_0+\tau}^+|^2 \, \d x  = 0,$$
	and this proves that $u \leq u_{\lambda_0+\tau}$ in $\Omega_{\lambda_0 +
		\tau}$ for any $0 < \tau <
	\bar{\tau}$ (provided $\bar{\tau}>0$ is sufficiently small). 
	Such a contradiction
	shows that
	$$ \lambda_0 = 0.$$
	Since the moving plane procedure can be performed in the same way
	but in the opposite direction, then this proves the desired symmetry
	result. The fact that the solution is in\-cre\-asing in the
	$x_1$-direction in $\{x_1 < 0\}$ is implicit in the moving plane
	procedure.
\end{proof}
 \begin{rem} \label{rem.Gammanonempty}
 As already mentioned in the Introduction, the approach
 adopted in the proof of Theorem \ref{thm:noSingularSym}
 \emph{cannot be reproduced} if $\Gamma\neq \emptyset$ and
 $p\in (1,2)$. 
 \vspace*{0.05cm}
 
 In fact, when $\Gamma\neq\emptyset$ the reflected function
 $u_{\lambda_0+\tau}$ \emph{is not} of class $C^1$ on $\overline{\Omega}_{\lambda_0+\tau}$;
 as a consequence, even if \eqref{eq:weight4} remains valid
 \emph{(}at least out of $R_{\lambda_0+\tau}(\Gamma)$, which has
 zero Lebesgue measure\emph{)}, \emph{we cannot deduce}
 from this estimate that
 $$1/\varrho = \left(1+|\nabla u|^2 + |\nabla u_{\lambda_0+\tau}|^2 \right)^{\frac{2-p}{2}}
 \in L^\infty(\Omega_{\lambda_0+\tau}),$$
 nor that $1/\varrho \in L^t(\Omega_{\lambda_0+\tau})$ for a sufficiently large $t$.
 This lack of information on the sum\-ma\-bi\-lity of $1/\varrho$ prevents us
 to apply the weighted Sobolev inequality \eqref{eq.SobolevTrud} in
 \eqref{eq:noSingfinal1}.
 
 On the other hand, when $p\in (1,2)$ we cannot use
 $\rho = |\nabla u|^{p-2}$ as a weight for the Sobolev
 inequality: this is due to the fact that, in general,
 we {cannot expect} that
 $$\rho \leq p(|\nabla u|+|\nabla u_{\lambda_0+\tau})^{p-2}
 +qa(x)(|\nabla u|+|\nabla u_{\lambda_0+\tau})^{q-2}.$$
 \end{rem}
\appendix
 \section{The strong comparison principle and the Hopf lemma.}
 We provide here a strong comparison principle
 and a Hopf-type lemma which apply to our context.
 Though these re\-sults seems to be very well-known as 
 consequences of rather general results by Serrin \cite{Serrin70}, 
 we explicitly write them here
 for future reference.
 \medskip
 
 In what follows, if $\mathcal{V}\subseteq\R^N$ is open
 and $v\in C^1(\mathcal{V})$, we say that
 $v$ satisfies
 $$-\mathrm{div} \left(p|\nabla v|^{p-2}\nabla v
  + q a(x)|\nabla v|^{q-2}\nabla v \right)\geq\,[\leq]\,\,0\quad
  \text{in $\mathcal{V}$}$$
 if, for every \emph{non-negative}
 test function $\varphi\in C^1_c(\mathcal{V})$, one has
 $$\int_{\mathcal{V}} \big(p|\nabla v|^{p-2}+qa(x)|\nabla v|^{q-2}\big)
	\langle \nabla v , \nabla
	\varphi\rangle\,\d x \geq\,[\leq]\,\,0.$$
 The function $a$ is assumed to be non-negative, bounded and of class $C^1$
 on $\mathcal{V}$.
 \begin{thm} \label{thm.SMP}
  Let $\mathcal{V}\subseteq\R^N$ be a connected
  open set, and let $v_1,v_2\in C^1(\mathcal{V})$ satisfy
  \begin{align*}
   & -\mathrm{div} \left(p|\nabla v_1|^{p-2}\nabla v_1
  + q a(x)|\nabla v_1|^{q-2}\nabla v_1 \right)\leq 0\quad \text{in $\mathcal{V}$},
  \\
  & -\mathrm{div} \left(p|\nabla v_2|^{p-2}\nabla v_2
  + q a(x)|\nabla v_2|^{q-2}\nabla v_2 \right)\geq 0\quad \text{in $\mathcal{V}$},
  \end{align*}
  respectively. We assume that
  \begin{equation} \label{eq.assnablav1neq}
   \text{$|\nabla v_1|\neq 0$ or $|\nabla v_2|\neq 0$ on the whole of $\mathcal{V}$.} 
  \end{equation}
  Then, either $v_1\equiv v_2$ or $v_1 < v_2$ throughout $\mathcal{V}$.
 \end{thm}
 \begin{proof}
  If $p\geq 2$, this result is a particular case of \cite[Theorem 1]{Serrin70}: in fact,
  following the notation of this cited theorem,
  our setting corresponds to the choices
  \begin{enumerate}
   \item ${A}(x,z,\xi)
   = p|{\xi}|^{p-2}{\xi}
   + q a(x)|{\xi}|^{q-2}{\xi}$ (for $x\in\mathcal{V},z\in\R$ and
   ${\xi}\in\R^N$); 
   \item $B(x,z,{\xi}) \equiv 0$.
  \end{enumerate}
  We explicitly notice that, since $q > p\geq 2$  and $a\in C^1(\mathcal{V})$,
  the function ${A}(x,z,{\xi})$ is of class $C^1$ on $\mathcal{V}\times\R\times\R^N$;
  moreover, since for every $(x,z,{\xi})\in\mathcal{V}\times\R\times\R^N$ we have
  \begin{align*}
   \partial_{{\xi}}{A}(x,z,{\xi})
   & = p|{\xi}|^{p-2}\Big[\mathrm{I}_N+\frac{p-2}{|{\xi}|^2}
   (\xi_i\cdot\xi_j)_{i,j = 1}^N\Big]
   +qa(x)|{\xi}|^{q-2}\Big[\mathrm{I}_N+\frac{q-2}{|{\xi}|^2}
   (\xi_i\cdot\xi_j)_{i,j = 1}^N\Big],
  \end{align*}
  it follows from assumption \eqref{eq.assnablav1neq}
  that \emph{at least one of the two matrices}
  $\partial_{{\xi}}{A}(x,v_1,\nabla v_1)$
  and $\partial_{{\xi}}{A}(x,v_2,\nabla v_2)$
  is positive definite for every $x\in \mathcal{V}$.
  \medskip
  
  If, instead, $p < 2$, the function $A$ is \emph{no longer}
  differentiable at $\xi = 0$; how\-ever, we claim that estimates
  (8) and (10) in \cite{Serrin70} are still satisfied in our case: more precisely,
  if $K\subseteq\mathcal{V}$ is compact, there exist
  constants $\mathbf{c}_1,\mathbf{c}_2 > 0$ such that, for any $x\in K$, one has
  \begin{align}
	& \big|A(x,v_2,\nabla v_2)-A(x,v_1,\nabla v_1)\big|
	\leq \mathbf{c}_1|\nabla v_2-\nabla v_1|; \label{eq.estimLowerSerrin} \\
	& \langle A(x,v_2,\nabla v_2)-A(x,v_1,\nabla v_1),\nabla v_2-\nabla v_1\rangle
	\geq \mathbf{c}_2|\nabla v_2-\nabla v_1|^2. \label{eq.estimUpperSerrin}
  \end{align}
  In fact, let us assume (to fix ideas) that $q\geq 2$; the case $1<q<2$
  can be faced analogously. Using the explicit expression
  of $A$ and the second estimate in \eqref{eq:inequalities}, we have
  \begin{equation} \label{eq.estimOrderfirst}
   \begin{split}
    & \big|A(x,v_2,\nabla v_2)-A(x,v_1,\nabla v_1)\big| \\
    & \qquad (\text{since $a$ is bounded on $\mathcal{V}$}) \\
    & \qquad
    \leq C_2\Big(p(|\nabla v_1|+|\nabla v_2|)^{p-2}+q\|a\|_{L^\infty(\mathcal{V})}
    (|\nabla v_1|+|\nabla v_2|)^{q-2}\Big)|\nabla v_2-\nabla v_1|.
   \end{split}
  \end{equation}
  Moreover, since $v_1,v_2\in C^1(\mathcal{V})$, $K\subseteq\mathcal{V}$
  is compact and $q\geq 2$, one has
  \begin{equation} \label{eq.estimsumq}
   (|\nabla v_1|+|\nabla v_2|)^{q-2} \leq \big(\max_{K}|\nabla v_1|+\max_K|\nabla v_2|\big)^{q-2}
  =: \kappa^{(1)}_{v_1,v_2}.
  \end{equation}
  Finally, since $p\in(1,2)$, 
  by crucially exploiting assumption \eqref{eq.assnablav1neq} we get
  \begin{equation} \label{eq.estimsump}
   (|\nabla v_1|+|\nabla v_2|)^{p-2}
  \leq \big(\inf_K|\nabla v_1|+\inf_K|\nabla v_2|)^{p-2} =: \kappa^{(2)}_{v_1,v_2} < \infty.
  \end{equation}
  Gathering together 
  \eqref{eq.estimOrderfirst}, \eqref{eq.estimsumq} and 
  \eqref{eq.estimsump}, we then obtain \eqref{eq.estimLowerSerrin}.
  
  As regards \eqref{eq.estimUpperSerrin}, we proceed essentially in the same way:
  using the explicit expression of the function
  $A$ and the first estimate in \eqref{eq:inequalities}, we have 
  \begin{equation} \label{eq.estimOrdersecond}
   \begin{split}
    & \langle A(x,v_2,\nabla v_2)-A(x,v_1,\nabla v_1),\nabla v_2-\nabla v_1\rangle\\
    & \qquad
    \geq C_1\Big(p(|\nabla v_1|+|\nabla v_2|)^{p-2}+q a(x)
    (|\nabla v_1|+|\nabla v_2|)^{q-2}|\Big)|\nabla v_2-\nabla v_1|^2\\
    & \qquad (\text{reminding that $a\geq 0$ on $\mathcal{V}$}) \\
    & \qquad \geq p\,C_1(|\nabla v_1|+|\nabla v_2|)^{p-2}|\nabla v_2-\nabla v_1|^2.
   \end{split}
  \end{equation}
  On the other hand, since $p\in(1,2)$, 
  again by exploiting assumption \eqref{eq.assnablav1neq} we get
  \begin{equation} \label{eq.estimsumpSecond}
   (|\nabla v_1|+|\nabla v_2|)^{p-2}
  \geq \big(\max_K|\nabla v_1|+\max_K|\nabla v_2|)^{p-2} = \kappa^{(1)}_{v_1,v_2} 
  \in (0,\infty).
  \end{equation}
  Gathering together \eqref{eq.estimOrdersecond} and 
  \eqref{eq.estimsumpSecond}, we then obtain \eqref{eq.estimUpperSerrin}.
  \medskip
  
  With \eqref{eq.estimLowerSerrin}-\eqref{eq.estimUpperSerrin} at hand, we
  can exploit the celebrated Harnack inequality established by Trudinger
  \cite{TruHarnack} and conclude exactly as in the proof of 
  \cite[Theorem 1]{Serrin70}.
 \end{proof}
 \begin{thm} \label{thm.Hopftype}
  Let $\mathcal{V}\subseteq\R^N$ be a fixed open set, and let $v\in C^1(\overline{\mathcal{V}})$
  satisfy
  $$-\mathrm{div} \left(p|\nabla v|^{p-2}\nabla v
  + q a(x)|\nabla u|^{q-2}\nabla v \right)\geq 0\quad
  \text{in $\mathcal{V}$}.$$
  We assume that $v < 0$ on $\mathcal{V}$ and that there exists a
  point $y\in\partial\mathcal{V}$
  such that $v(y) = 0$. Then, if $\mathcal{V}$ satisfies the \emph{interior cone
  condition} at $y$ and $|\nabla v|\neq 0$ on $\mathcal{V}$, one has
  $$\nabla v(y) \neq 0.$$
 \end{thm}
 \begin{proof}
  First of all, by \cite[Theorem 3]{Serrin70} (applied here with $u\equiv 0$) we know that
  the zero of $v$ ad $y$ is of finite order, say $m\in\mathbb{N}\cup\{0\}$.
  In fact, following the notations
  in this cited theorem, our context corresponds to the choices
  \begin{enumerate}
   \item ${A}(x,z,{\xi})
   = p|{\xi}|^{p-2}{\xi}
   + q a(x)|{\xi}|^{q-2}{\xi}$ (for $x\in\mathcal{V},z\in\R$ and
   ${\xi}\in\R^N$); 
   \item $B(x,z,{\xi}) \equiv 0$.
  \end{enumerate}
  We explicitly notice that, when $p\geq 2$, the function $A$ is continuously
  differentiable on the whole of $\mathcal{V}\times\R\times\R^N$; when $p\in (1,2)$, instead,
  the function $A$ is no longer differentiable ad $\xi = 0$ but we have at our disposal
  the estimates \eqref{eq.estimLowerSerrin}-\eqref{eq.estimUpperSerrin}.
  
  To conclude the demonstration
  it suffices to observe that, 
  since in our case we have $B(x,z,{\xi}) \equiv 0$,
 a closer inspection to the proof of \cite[Theorem 3]{Serrin70} shows that
 $$m = 0;$$
 from this, we immediately deduce that
 $\nabla v(y) \neq 0$, as desired.
 \end{proof}

\end{document}